\documentclass{siamltex}

\usepackage[cp1251]{inputenc}

\usepackage{lipsum}
\usepackage{amsfonts}
\usepackage{graphicx}
\usepackage{epstopdf}
\usepackage{algorithm}
\usepackage{algorithmic}
\usepackage[normalem]{ulem}
\usepackage{amsmath}
\usepackage[font=scriptsize]{caption}
\usepackage[]{color}
\usepackage{soul}
\usepackage[space,noadjust,nocompress]{cite}

\usepackage[font=scriptsize]{subcaption}
\ifpdf
  \DeclareGraphicsExtensions{.eps,.pdf,.png,.jpg}
\else
  \DeclareGraphicsExtensions{.eps}
\fi

\newtheorem{prop}{Proposition}
\newtheorem{example}{Example}[section]

\usepackage{amsopn}

\DeclareMathOperator{\trace}{tr}

\title{A posteriori superlinear convergence bounds for \\ block conjugate gradient \thanks{Christian E. Schaerer: his research is partially supported by PRONII and by CONACYT-PY under project 14-INV186. Daniel B. Szyld: his research is supported in part by the U.S. National Science Foundation under grant DMS-1418882 and U.S. Department of Energy under grant DE-SC-0016578. Pedro Torres acknowledges the financial support of Scholarship-FEEI-CONACYT-PROCIENCIA. 
}}

\author{
   Christian E. Schaerer\thanks{Polytechnic School, National University of Asuncion, San Lorenzo, Central, Paraguay. 
   P.O.Box: 2111 SL. 
  (Email: {cschaer@pol.una.py}, {pjtorres@pol.una.py}). Corresponding author: P. J. Torres.}
   \and
   Daniel B. Szyld\thanks{Department of Mathematics, Temple University
            (038-16), 1805 N. Broad Street, Philadelphia, PA 19122-6094,
            USA. (Email: {szyld@temple.edu}).}
   \and
   Pedro J. Torres\footnotemark[2]
   }

\begin{document}

\maketitle

\renewcommand{\thefootnote}{\fnsymbol{footnote}}






\begin{abstract} 
 In this paper, we extend to the block case, the {\it a posteriori} bound showing superlinear 
convergence of Conjugate Gradients developed in [J.\ Comput.\ Applied Math., 48 (1993), pp. 327-341]; 
that is, we obtain similar bounds, but now for block Conjugate Gradients. We also present a series of computational experiments, illustrating the validity of the bound developed here, as well as the bound from [SIAM Review, 47 (2005), pp. 247-272] using angles between subspaces. 
Using these bounds, we make some observations on the onset of superlinearity, and how 
this onset depends on the eigenvalue distribution and  the block size. 
\end{abstract}



\section{Introduction}\label{sec:introduction}
When solving numerically linear systems of the form $Ax =b$, for a 
given right hand side $b$, and when  $A$ is a large $n \times n$  sparse symmetric positive
definite (s.p.d.) matrix, the method of choice is
conjugate gradient (CG) \cite{hestenes1952methods}.
It is well known that CG exhibits superlinear convergence. 
In this context the term {\em  superlinear} is understood to mean that
the $A$-norm of the error is monotonically decreasing at first in a linear manner, 
but at some point the slope of this graph may change and become faster. In other words,
the method is faster than linear; see, e.g.,
Figure~\ref{fig:factor_vds_ss_2} in Section~\ref{sec:experiments_bcg}. 
Several authors studied this phenomenon; see, e.g., 
\cite{beckermannsuperlinear2002,SLEIJPEN1996233,vanderSluis1986, van1987convergence}.
In particular, they studied in an {\em a posteriori} manner when one would expect
the change from a linear regime, to a superlinear regime (i.e., a steeper slope), when this
occurs.\footnote[7]{We note that in some cases,the superlinear regime may not 
start before convergence is reached. In those cases, the convergence is just linear}
See also \cite{Axelsson2002, Elman00, Gergelis2014, Campbell96,Strakos1991,Winther1980} for
different analyses including studies on the effect of round-off errors.

When one has $s > 1$ right hand sides, i.e., when one wishes to solve a block system of the form
  \begin{equation}\label{eq:linear_system11}
   A{\bf X}={\bf B},
  \end{equation}
where $\bf{X}$ and $\bf{B}$ are $n\times s$ skinny-tall matrices (block vectors)  with $s\ll n$, 
then, the block conjugate gradient (block CG) method can be considered for its solution.
This block method was introduced  by O'Leary in 1980  \cite{o1980block}, and can be used either
when  one has multiple right-hand sides, or when one wishes to accelerate the convergence of CG
by using a richer space (usually by adding random vectors to the original right-hand side $b$ to produce ${\bf B}$).
We have this latter situation very much in mind in our investigations.
  
In this paper, we extend the theory developed by van der Sluis and van der Vorst 
\cite{vanderSluis1986,van1993superlinear} in order to explain
the superlinear convergence of CG in the block case. Their analysis uses
spectral information (and Ritz values) to bound the convergence of CG by that of a comparison
method which commences with an initial vector projected onto a subspace where certain components 
have been deflated. This projection is based on polynomial expressions describing
the CG method. We call these 
new {\em a posteriori} convergence bounds, {\em spectral-based bounds} for short,
and they are discussed in Section~\ref{seq:spectrum}.

There is another set of convergence bounds developed by
Simoncini and Szyld \cite{Simoncini2005,simoncini2013superlinear} where angles 
between subspaces are considered.  We call these bounds,  {\em subspace-based bounds} for short,
and they are reviewed in Section~\ref{seq:subspacemodel}.

We illustrate the effectiveness of these bounds with a series of numerical experiments
in Section~\ref{sec:experiments_bcg}.
It can be observed that both bounds capture well the slope of the block error or residual
norm (in the appropriate block norm). 
As one would expect intuitively, the larger the size of the block~$s$, the faster the convergence,
and in particular, the onset of superlinearity (the point where the slope changes in the
convergence curve) occurs earlier; cf. \cite{Kubinova2020, Simoncini2005}. 
The {\em a posteriori} bounds described in this paper
reproduce properly this onset of superlinarity, both in cases of single eigenvalues or
with multiple eigenvalues (or clusters).
These bounds explain the observed
superlinear convergence behavior of block CG: the onset of superlinarity
occurs when the block Krylov subspace is close to an invariant subspace of
the matrix $A$ (subspace bound), or when the Ritz values (or latent roots
of the residual polynomial) are close to the eigenvalues of the matrix $A$
(spectral bound).

   Throughout this paper, calligraphic  letters ${\cal A}$, ${\cal H}$ and ${\cal Z}$ 
   denote square $s\times s$ matrices and upper case letters $A$, $V$, $D$ and $H$ denote 
   square $n\times n$ matrices and ${\bf R}$, ${\bf X}$, ${\bf U}$, ${\bf W}$ and ${\bf B}$ 
   denote rectangular $n\times s$ matrices, matrix polynomials are denoted with upper case Greek letters 
   (e.g., $\Phi$, $\Omega$, $\Xi$). 
   Scalars are denoted with lower case Greek letters (e.g., $\lambda$, $\gamma$, $\alpha$) and integers are 
   normally denoted as $i, j, k, m$ and $n$. A subscript on a matrix, vector, or scalar denotes an iteration 
   number. 
   For simplicity, we consider only real matrices, although the generalization to complex matrices is direct.
     

\section{Block CG}\label{sec:polyblock_bcg}
In this section we review the CG method and its block counterpart. We present 
their formulation as minimization problems. We discuss the block Lanczos method,
and its formulation using a matrix-polynomial approach.

At each step $m$, CG searches for the approximate solution $x_m$ in the
Krylov subspace ${\cal K}_m(A,r_0) = \{ \sum_{k=0}^{m-1} c_k A^k r_0: c_k \in \mathbb{R} \}$
shifted by the initial vector $x_0$,
and such that the $A$-norm of the error is minimized, i.e.,
$x_m$ is such that 
\begin{equation} \label{CG:eq}
\|x_* - x_m \|_A = \min_{x \in x_0 + {\cal K}_m} \|x_* - x \|_A  = 
\min_{d \in A  {\cal K}_m} \|r_0 -d \|_{A^{-1}}
= \|r_m\|_{A^{-1}} ,
\end{equation}
where $r_m = b - Ax_m$ is the $m$th residual vector, 
$x_*$ is the solution of $Ax=b$, and the $A$-norm is defined as usual
as $\|v\|_A = \langle v, A v \rangle^{1/2}$ with the latter being the Euclidean inner product;
see, e.g., \cite{golub2012matrix,greenbaum1997iterative,saad2003iterative,Simoncini.Szyld.07}.
The latter two equalities in \eqref{CG:eq} indicate that minimizing the $A$-norm
of the error is mathematically equivalent to minimizing the $A^{-1}$-norm of
the residual. Of course the latter is never computed explicitly, but as we shall see,
it is a useful way of looking at the method.
     
Similarly, one can construct the block Krylov subspace as
  \begin{equation} \label{bKrylov:eq}
          \mathbb{K}_m(A, {\bf R}_0):=\Bigg \{ \sum _{k=0} ^{m-1} A^k{\bf R}_0 {\cal C}_k: 
                {\cal C}_k \in \mathbb{S} \Bigg \},
  \end{equation}
where the initial block residual is  ${\bf R}_0 := {\bf B} - A{\bf X}_0$,
with ${\bf X}_0$ being the initial block vector and 
$\mathbb{S} \subseteq \mathbb{R}^{s\times s}$ is a subspace containing the identity ${\cal I}_s$ and 
closed under matrix multiplication and transposition;
see, e.g.,  \cite{greenbaum1997iterative, Frommer2017, Frommer2020,GutknechtKrylov, saad2003iterative}.\footnote{%
We point out that other versions of block Krylov spaces have been used in the literature,
e.g., in \cite{ELGUENNOUNI2004243}, but in this paper we will use the above definition,
known as {\em classical}. See also \cite{GUTKNECH}.}

At each step $m$, block CG searches for the approximate solution on 
       the $m$th block Krylov subspace \eqref{bKrylov:eq}, shifted by ${\bf X}_0$, 
such that it minimizes the $A$-norm of the (block) error. Let ${\bf X}_*$ be the
solution of \eqref{eq:linear_system11}, then the approximation ${\bf X}_m$ and
the corresponding residual ${\bf R}_m := {\bf B} - A{\bf X}_m$ satisfy
    \begin{equation}\label{eq:residual_block}
     \min_{{\bf X} \in {\bf X}_0 + \mathbb{K}_m }  \| {\bf X}_* - {\bf X} \|_A =
              \min_{{\bf D} \in A\mathbb{K}_m } \| {\bf R}_0 - {\bf D} \|^2_{A^{-1}} =
             \| {\bf R}_m \|^2_{A^{-1}} ,
     \end{equation}
where as usual the norm of block vectors is the Frobenious norm, so that the $Z$-norm
(for $Z=A$ or $Z = A^{-1}$) 
here is defined with the inner product which induces the Frobenious norm, that is,
$$ 
       \left\langle {\bf V}, {\bf W} \right\rangle_{Z} := \trace({\bf V}^T Z {\bf W}) 
       = \sum^{s}_{i=1} \left\langle v_i, w_i \right\rangle_{Z}, 
$$ 
where ${\bf V}=\left[ v_1, \ldots, v_s \right]$ and ${\bf W}=\left[ w_1, \ldots, w_s \right]$.
We denote the solution of the minimization problem  on the right in
(\ref{eq:residual_block}) by~${\bf D}_{m}$.

\subsection{Block Lanczos}
The block Lanczos procedure is the block version of the Lanczos method. It produces
a block basis
of the block Krylov susbpace $\mathbb{K}_m(A, {\bf R}_0)$ which we collect in the matrix
   $W_m=[{\bf U}_0, {\bf U}_1, \dots, {\bf U}_{m-1}]$.
It proceeds with a three-term recurrence as follows.
Let ${\bf R}_{0}={\bf U}_0 {\cal B}_0$ be the \textit{QR}-factorization of ${\bf R}_0$, where ${\bf U}_0$ 
   and  ${\cal B}_0$  are real $n \times s$ and $s \times s$ matrices, respectively. 
Then, the sequence  of block vectors $\{ {\bf U}_i\}, i=0, 1,2\ldots,m$, 
which have orthonormal columns and are mutually orthogonal are constructed by
the  three-term recurrence 
 \begin{equation}\label{eq:block_lanczos_iter_1}
{\bf U}_{i+1} {\cal B}_{i+1} = {\bf M}_i = A {\bf U}_i - {\bf U}_i {\cal A}_i - {\bf U}_{i-1} {\cal B}^{T}_{i},
 \end{equation}
   where ${\bf U}_{-1}={\bf 0}$, ${\cal A}_i={\bf U}^*_i A {\bf U}_i$ and ${\bf U}_{i+1}{\cal B}_{i+1}$ is the \textit{QR}-factorization of ${\bf M}_i$
   with ${\bf U}_{i+1}$ orthogonal and ${\cal B}_{i+1}$ upper triangular; see, e.g., \cite{golub2012matrix}.  
Thus, at each iteration step $m$, the following matrix relation, called the block Lanczos relation,
     holds
  \begin{equation}\label{eq:block_lanczos_iter}
      A [ {\bf U}_0, {\bf U}_1, \ldots, {\bf U}_{m-1}]= [ {\bf U}_0, {\bf U}_1, \ldots, {\bf U}_{m-1}] T_m +  {\bf U}_{m} {\cal B}_m {E^T_m} ,
  \end{equation}
where ${E_m=[0_s,0_s,\ldots, {\cal I}_s]^T}\in \mathbb{R}^{s m \times s}$ and $T_m$ is the 
block tridiagonal matrix 
of dimension  ${ms \times ms}$
  \begin{equation}
         T_m = \left( \begin{matrix} {\cal A}_0 & {\cal B}^T_1 &  \\ {\cal B}_1 & {\cal A}_1 & \ddots \\   
              & {\cal B}_2 & \ddots \\ &  & \ddots & {\cal A}_{m-2} & {\cal B}^T_{m-1} \\
      & &  & {\cal B}_{m-1} & {\cal A}_{m-1}
  \end{matrix} \right).
  \end{equation}
The expression 
   (\ref{eq:block_lanczos_iter}) can be rewritten as 
     \begin{equation}\label{eq:matrix_H}
        A \mathit{W}_m= \mathit{W}_{m+1} \mathit{\tilde{T}_m}.
   \end{equation}
   where $\mathit{\tilde{T}_m}$
    is a real block tridiagonal  matrix of dimension  ${(m+1)s \times ms}$, 
and 
$\mathit{T}_m = [\mathit{I}_{ms}, 0]\mathit{\tilde{T}_m}$.
   
  The approximate solution of the linear systems produced by block CG  at the $m$th iteration is given by 
    \begin{equation}\label{sol:Xm}
     {\bf X}_m={\bf X}_0 + W_m Y_m,
    \end{equation}
   where ${ Y}_m=[y^T_1,\ldots,y^T_m]^T$ with 
    $y_i \in \mathbb{R}^{s\times s}$ is obtained by solving the equation 
\begin{equation} \label{eq:ym_definition}
    T_m Y_m= {\cal B}_0 E_1
\end{equation}  
with  ${E_1=[{\cal I}_s, 0_s, \ldots, 0_s]^T} \in \mathbb{R}^{s m\times s}$. 
    Since   ${\bf R}_0 = {\bf U}_0 {\cal B}_0$, then equation  (\ref{eq:ym_definition})
    is equivalent to $W_mT_mY_m= {\bf R}_0$. In addition,   
      at each block iteration it is possible to compute the residual of the solution 
    without explicitly computing the solution, using the expression 
\linebreak
     ${\bf  R}_m =  - {\bf U}_m {\cal B}_m E_m Y_m $ 
    which is obtained as follows (see, e.g., \cite{Frommer2017}, \cite{saad2003iterative})
     \begin{align}
           {\bf  R}_m & =   {\bf B} - A {\bf X}_m  &    \mbox{by definition of ${\bf R}_m$} \nonumber \\
                             & =   {\bf B} - A({\bf X}_0 + W_mY_m) &  \mbox{by expression (\ref{sol:Xm})} \nonumber  \\
                             & =   {\bf R}_0 - AW_mY_m & \mbox{ by definition of ${\bf R}_0$} \nonumber \\
	                   & =   {\bf R}_0 - W_mT_mY_m - {\bf U}_m {\cal B}_m E_m Y_m &   
\mbox{by identity(\ref{eq:block_lanczos_iter})} \nonumber  \\
	                & =   - {\bf U}_m {\cal B}_m E_m Y_m &   \mbox{by identity (\ref{eq:ym_definition}),  
	                \label{eq:51}} 
   \end{align} 
and  notice that   
     ${\bf R}_m={\bf R}_0- {\bf D}_m$ and ${\bf D}_m=AW_m Y_m$ in 
     (\ref{eq:residual_block}).
  
  
An implementation of the block Conjugate Gradient method is described 
in Algorithm \ref{alg:blockcg} below.
We note that in this algorithm, assuming exact arithmentic, one has
the orthogonality properties ${\bf R}^T_k {\bf R}_j=0$ and {${\bf P}^T_k A {\bf P}_j=0$}  with $j\neq k$, 
and as long as the matrices {${\bf P}_k$} and {${\bf R}_k$} retain full rank, the algorithm is well defined. 
  \begin{algorithm}[thb]
 \caption{Block Conjugate Gradient - block CG \cite{o1980block}.  \label{alg:blockcg}}
 \begin{algorithmic}[1]
 \STATE Given an initial approximation ${\bf X_0}$ to the solution matrix ${\bf X^*}$;  
 \STATE Set ${\bf R}_0={\bf B}-A {\bf X}_0$, ${\bf P}_0 = {\bf R}_0$;
 \FOR    {$k=0,1,2 \dots $}
 \STATE {${\cal D}_k = ({\bf P}^{T}_k A {\bf P}_k)^{-1} {\bf R}^T_k {\bf R}_k$}
 \STATE {${\bf X}_{k+1} = {\bf X}_k + {\bf P}_k {\cal D}_k$ }
 \STATE {${\bf R}_{k+1} = {\bf R}_k - A{\bf P}_k {\cal D}_k$ }
 \IF {convergence conditions are satisfied}
\STATE break;
\ENDIF
 \STATE {${\cal P}_{k}= ({\bf R}^T_k {\bf R}_k)^{-1}{\bf R}^T_{k+1} {\bf R}_{k+1}$ }
  \STATE {${\bf P}_{k+1}=({\bf R}_{k+1}+{\bf P}_k {\cal P}_{k})$ }
  \ENDFOR
  \end{algorithmic}
  \end{algorithm} 
  
  
\subsection{Block residual polynomial}
The spectral bounds are based on the analysis of the block residual polynomials.
In this section we bring forth some preliminary definitions, identities and properties
of matrix polynomials, and in particular the block residual polynomial.
   
Let $\mathbb{P}_{m,s}$  be the space of matrix-valued polynomials with elements of the form 
    \begin{equation}
    \Upsilon_m(\eta)=\sum^{m}_{i=0} \eta ^i {\cal C}_i
    \end{equation}
     where $\eta \in \mathbb{R}$ and  
   ${\cal C}_i$ are real ${s\times s}$ matrices.  
We recall the operation introduced in \cite{mdKent89}, 
   \begin{equation}\label{Graag:operator}
   \Upsilon_m(A) \circ {\bf X}= \sum^{m}_{i=0} A^i {\bf X} {\cal C}_i,
   \end{equation}
    where $A$ is any  $n\times n$ matrix, ${\bf X}$ is a block $n \times s$  vector and `$\circ$' is 
    called the Gragg operator.
    
    {Denote by $\mathbb{G}_{m,s} \subset \mathbb{P}_{m,s}$}  the 
    subspace of matrix-valued polynomials with elements of the form 
    $\Upsilon_m(\eta) = {\cal I}_s - \sum^{m-1}_{i=0}\eta ^{i+1}{\cal C}_{i}$, 
i.e., the polynomias such that $\Upsilon_m(0) = {\cal I}_s$.
 Hence, using the nomenclature (\ref{Graag:operator}) and the subspace $\mathbb{G}_{m,s} $, the residual  
    ${\bf R}_m$ of block CG at the $m$th iteration can be expressed in terms of
     matrix polynomials as 
 \begin{equation}\label{eq:residual_matrix_polynomial1}
     {\bf R}_m = \Upsilon _m(A)\circ {\bf R}_0={\bf R}_0 - \sum^{m-1}_{i=0} A^{i+1}{\bf R}_0 {\cal G}_i,
 \end{equation}
   where $ \Upsilon_m \in \mathbb{G}_{m,s} $,  and ${\cal G}_{i}$ are $s\times s$ matrices, $i=1,\ldots,m-1$. 

    Consequently, the variational formulation of block CG (\ref{eq:residual_block}) 
  can be expressed as follows using matrix-value polynomials 
     with $A$  and ${\bf R}_0$  as arguments 
\begin{equation}\label{eq:bcg_pol_21}
             \|{\bf  R}_m \|_{A^{-1}} = \min_{\Upsilon _m \in {\mathbb{G}}_{m,s}  } 
             \| \Upsilon _m(A)\circ {\bf R}_0\|_{A^{-1}}
             = \| \Phi_m(A)\circ {\bf R}_0\|_{A^{-1}} ,
  \end{equation}
where $\Phi_m(\eta ) \in \mathbb{G}_{m,s} $ is the solution of the minimization problem.
   
The three-term recurrence (\ref{eq:block_lanczos_iter}) of Block Lanczos 
can also be written in matrix polynomial form.
   Each matrix ${\bf U}_i$ in the recurrence (\ref{eq:block_lanczos_iter}) is
a linear combinations of matrices $ A^{i}{\bf R}_0 \in \mathbb{K}_m (A, {\bf R}_0)$ for $i=0,1,\ldots,m-1$.
Therefore, we can set ${\bf U}_i=\Gamma_i(A) \circ {\bf R}_0$, and thus, (\ref{eq:block_lanczos_iter}) 
   can be rewritten in matrix polynomial  form as 
    \begin{equation}\label{eq:H_pol_matrix1}
        \eta  P_{m-1}(\eta)=P_{m-1}(\eta) T_m +  \Gamma_m(\eta) {\cal B}_m E_{m},
    \end{equation}
     where $P_{m-1}(\eta):=[\Gamma _0(\eta), \Gamma _1(\eta),\ldots, \Gamma _{m-1}(\eta)]$
     with  $\Gamma_i \in \mathbb{P}_{i,s}$; see, e.g., \cite{mdKent89,SIMONCINI1996_Ritz}.

    It can be observed from (\ref{eq:H_pol_matrix1}) that $\det(\lambda I - T_m)=0$ 
     if and only if $\det(\Gamma _m(\lambda))=0$. Therefore, the eigenvalues of $T_m$ 
     are the latent roots of $\Gamma _m(\lambda)$, hence coinciding with the Ritz values of $A$ 
     associated with $\mathbb{K}_m(A, {\bf R}_0)$. In addition, the matrix 
     polynomials $\Gamma _m ( \lambda )$ and $\Phi_m ( \lambda )$ represent the block CG process dynamics 
     but from different perspective \cite{Frommer2017, SIMONCINI1996_Ritz}. In the 
following proposition we show that the latent roots of these two
      matrix polynomials are the same.
  \begin{prop}
    The latent roots of the block CG polynomial $\Phi_m(\eta)$ coincide with the latent 
    roots of $\Gamma_m(\eta)$. Hence they are also the eigenvalues of $T_m$ and the Ritz 
    values of the matrix $A$ associated 
    with $\mathbb{K}_m(A, {\bf R}_0)$.
 \end{prop}

\begin{proof}
   From (\ref{eq:51})  and using matrix value representation, the block 
   residual can be expressed as 
  \begin{equation*}
      {\bf R}_m= -(\Gamma_m(A) \circ {\bf U}_0){\cal B}_m E_m Y_m.
   \end{equation*}
   Using (\ref{eq:residual_matrix_polynomial1}) we can arrive to the following equality in matrix polynomial form
   \begin{equation*}
       \Phi_m(\eta)=-\Gamma_m(\eta) {\cal B}_m E_m Y_m,
   \end{equation*}
   then the latents roots of $\Phi_m(\eta)$ and $\Gamma_m(\eta)$ are the same.
  \end{proof}
 

\section{\textit{A posteriori} models for block CG}\label{sec:posteriori_bcg}
We are ready to present the two \textit{a posteriori} models which explain the
superlinear behavior of block CG. In the subspace-based model, the bound is based
on the angle (or gap) between the block Krylov subspace and an invariant subspace of $A$.
The $A$-norm of the error, or equivalently, the $A^{-1}$-norm of the (block) residual
is bounded by the residual norm of another CG process in which the components
of the corresponding eigenvectors have been deflated.
In the spectral bound, we have a similar comparison CG process, and the bound
is based on the difference between the eigenvalues and the Ritz values.

Let $\lambda_1, \ldots, \lambda_n $ and $v_1, \ldots, v_n$ be the eigenvalues and eigenvectors associated with the matrix $A$, chosen such that they form an orthonormal basis. Then $A=V \Lambda V^T$ is a spectral decomposition of matrix $A$ with $V=\left[v_1,\ldots,v_n\right]$ and $\Lambda = \diag\left[\lambda_1,\ldots,\lambda_n\right]$. 

The comparison CG process is a residual sequence, in which the components on the chosen invariant subspace have been deflated \cite{Simoncini2005,van1987convergence}. To define this more precisely let $\Pi_Q$ be a spectral projector onto an invariant subspace   $\mathbb{R}(Q)$ of the matrix $A$ with dimension $k$\footnote{Here and elsewhere in 
the paper $\mathbb{R}(Q)$ denotes the range of the matrix $Q$. We mention that in\cite{Simoncini2005} the invariant subspace must be simple, i.e., such that there is a complementary subspace which is also invariant. Here, this condition is fulfilled automatically since the invariant subspaces consist of linear combinations of orthogonal eigenvectors.}, where $Q$ is an $n\times k$ matrix whose columns are $k$ eigenvectors of $A$.
The spectral projector is constructed as $\Pi_Q=QQ^T$ and in this case it is also an orthogonal projector with respect to the $A^{-1}$-inner product $\left\langle \cdot, \cdot \right\rangle_{A^{-1}}$ since the matrix $A^{-1}$ commutes with its spectral projector, i.e., $\Pi_Q A^{-1}=A^{-1} \Pi_Q$.

The block comparison process is defined as a CG process which commences with
$\bar{\bf R}_0=(I-\Pi_Q){\bf R}_m$, i.e., with the initial residual being the $m$th block residual of the original CG process, but having its components in $\mathbb{R}(Q)$ deflated,
i.e., such that
  
   \begin{equation}\label{eq:optimal_cg}
              \| \bar{\bf R}_j \|_{A^{-1}} = 
              \min_{{\bf D}\in A\mathbb{K}_j(A,\bar{\bf R}_0)} \| \bar{\bf R}_0 - {\bf D}\|_{A^{-1}} . 
     \end{equation}


\subsection{Subspace-based bound for block CG\label{seq:subspacemodel}} 
We present an {\em a posteriori bound} which is a special case of those presented in \cite{Simoncini2005}.
As indicated above, the bound is obtained by considering a comparison CG process
defined by (\ref{eq:optimal_cg}). 
It is an {\em a posteriori bound} since it assumed that
the approximate solution at the $m$th step ${\bf X}_m$ is known, 
   and consequently, the corresponding residual ${\bf R}_m$. 
We summarize the subspace-based bound in the following theorem, where the norm of
the residual $j$ steps after the $m$th step is bounded by the $j$-th residual norm
of the comparison process. The factor in the bound depends on
the angle between the Krylov subspace, and the invariant subspace $\mathbb{R}(Q)$, the same
subspace used for the deflation.


\begin{theorem}[Subspace-based bounds for block Conjugate Gradient~\cite{Simoncini2005}]\label{thm:sbb_CG_1} 
Consider an {$n \times k$} real matrix $Y$, whose columns are a basis of a $k$-dimensional 
subspace of $A\mathbb{K}_m(A,{\bf  R}_0)$. Let $Q$ be an $n\times k$ 
matrix whose columns are $k$ eigenvectors of $A$ and $\Pi_Q$ a spectral projector 
onto the invariant subspace $\mathbb{R}(Q)$. 
Let $\Pi_Y$ be the $A^{-1}$-orthogonal projector onto $\mathbb{R}(Y)$, and 
let $\gamma _m = \| (I-\Pi_Y)\Pi_Q \|_{A^{-1}}$. Then 
              \begin{eqnarray}\label{eq:cota_ss_cg11} 
  \| {\bf R}_{m+j}\|_{A^{-1}} & \leq & \min_{  {\bf D} \in A\mathbb{K}_j(A, {\bf R}_m)} 
           \left\{ \|(I-\Pi_Q)({\bf R}_m - {\bf D})\|_{A^{-1}}  \right. \\ \nonumber 
                       & & 
 \qquad \qquad \qquad \qquad \left.  +\gamma _m \|\Pi_Q({\bf R}_m - {\bf D})\|_{A^{-1}}  \right\} \\  
                              \label{eq:cota_ls_cg11}
    &\leq & \sqrt[]{2} \min_{{\bf D} \in A\mathbb{K}_j(A, {\bf R}_m)} \left\lVert \begin{bmatrix}(I-\Pi_Q)\\
    \gamma _m \Pi_Q   \end{bmatrix} ({\bf R}_m - {\bf D})\right\rVert _{\star},
               \end{eqnarray}

   \noindent where 
   $\| \cdot \|_{\star}$ is an induced vector norm from the following inner product. 
   Let $u_i,v_i \in \mathbb{R}^n$, $i=1,2$; then, if $u^T=\left[u^T_1, u^T_2\right]$, $v^T=\left[v^T_1,v^T_2\right]$,   
   $\left\langle u,v\right\rangle_{\star}=\left\langle u_1,v_1\right\rangle_{A^{-1}} + \left\langle u_2,v_2\right\rangle_{A^{-1}}$.
   
   \end{theorem}

The quantity $\gamma_m$ corresponds to the angle between a subspace of $A\mathbb{K}_m(A,{\bf  R}_0)$ and
the invariant subspace since for symmetric $A$,
\begin{equation}\label{eq:gamma_identity1}
   \gamma _m = \| (I-\Pi_Y)\Pi_Q \|_{A^{-1}} = \| \Pi_Y - \Pi_Q \|_{A^{-1}}= \sin \varphi \leq  1,
\end{equation}
 where $\varphi$  is the maximum canonical angle between $\mathbb{R}(Y)$ and 
  $\mathbb{R}(Q)$; see, e.g., {\cite[p.~56]{kato1995perturbation}} and {\cite[p.~92]{steward90}} for details. 
   
   The computation of the upper bound is possible using the expression (\ref{eq:cota_ls_cg11}), which is a
   reformulation of  (\ref{eq:cota_ss_cg11}) as a least squares problem of dimension $2n$. 
    The least-squares problem (\ref{eq:cota_ls_cg11}) is well posed as long 
    as  {$\mathbb{R}(Q)\cap A\mathbb{K}_j(A, {\bf R}_m)=\{0\}$}.

We note that this theorem is very general and applies to {\em any} space  {$\mathbb{R}(Y)$} which is a subspace
of the Krylov subspace, and {\em any} invariant subspace {$\mathbb{R}(Q)$}. Of course for our bound to be
meaninful, one would take an appropriate invariant subspace close to the Krylov subspace.
It is well known that the eigenvalues which are captured first as the iterations $k$ progress,
are those at the end of the spectrum; see, e.g., \cite{vanderSluis1986}. 
Thus, in our numerical experiments we choose
eigenvectors corresponding to $k_1$ eigenvalues in the lower part of the spectrum and/or
$k_2$ eigenvalues in the upper part of the spectrum.

\subsection{Spectral-based bound for block CG\label{seq:spectrum}} 
In this section, we present a convergence bound 
for the block CG algorithm based on spectral information. 
The aim is to generalize to the block case, the bounds developed in \cite{van1987convergence}.
In this process, we benefited from the background material provided in \cite{mdKent89}. 
The bounds are obtained by considering the optimality property~(\ref{eq:residual_block})
of block CG and constructing a comparison process using an auxiliary matrix polynomial.
The bound uses the Ritz values, which, as we have shown, 
coincide with the latent roots of the residual polynomial $\Phi_m(\lambda)$,  and they 
converge to the eigenvalues of matrix $A$ for a sufficiently large $m$ \cite{SIMONCINI1996_Ritz}.

We begin by recalling the following result.
    \begin{lemma}\label{lem:block_poly_1}
            {\rm \cite{mdKent89}} 
Let $\Upsilon _l (\eta)=\sum^{l}_{i=0}\eta^i {\cal C}_i$ be any matrix polynomial 
            where  ${\cal C}_i$ are $s\times s$ matrices, and let $A$ be a $n\times n$ matrix, ${\bf Z}$ a $n\times s$ 
            matrix and $S$  any invertible matrix of order~$n$. Then the following result holds,
         \begin{equation}
                  \Upsilon _l (A) \circ {\bf Z}= S \Upsilon _l (S^{-1}AS)\circ (S^{-1}{\bf Z}).
         \end{equation}
     \end{lemma}
  
  Using Lemma \ref{lem:block_poly_1}, 
we can write the norm of the block residual 
(\ref{eq:bcg_pol_21}) 
in terms of polynomials evaluated at
the eigenvalues of the matrix, as shown  in
  the following lemma (adapted  from \cite[Ch.~4, expression (4.7)]{mdKent89}), and 
it will be used in the theorems that follow. 
  

   \begin{lemma}\label{lem:block_poly_2}
Let  $A=V \Lambda V^T$ be a spectral decomposition.
Let $\left[w_1, w_2, \ldots, w_n \right]^T=V^TR_0$ be the components of the block initial residual on 
       the eigenbasis, where $w_i$'s ($i=1,\ldots,n$) are $s\times 1$ matrices. Consider 
${\bf R}_m=\Phi_m(A)\circ {\bf R}_0$ the residual of block CG at iteration $m$. Then,
     \begin{eqnarray}\label{eq:block_wighted}
            \| {\bf R}_{m}\|^2_{A^{-1}}& = \trace\left({\bf R} ^T_m A^{-1} {\bf R}_m \right) 
            = \trace\left[\sum^n_{i=1} \frac{1}{\lambda_i}    
           \Phi^T_m(\lambda_i) w_i w^T_i \Phi_m(\lambda_i)\right].		
\end{eqnarray}
  \end{lemma}
    \begin{proof}
    Using Lemma \ref{lem:block_poly_1} and noting that $VV^T=I_n$ we can write 
    \begin{eqnarray*}
						{\bf R}_m &= & \Phi_m(A)\circ {\bf R}_0 \\
						         & = & V\Phi_m(V^TAV)\circ V^T {\bf R}_0
						          = V \Phi_m(\Lambda)\circ \begin{bmatrix}w^T_1 \\ 
														w^T_2 \\ 
														\vdots \\ 
														w^T_n \end{bmatrix}  
													=  V \begin{bmatrix} w^T_1 \Phi_m(\lambda_1)  \\ 
														w^T_2 \Phi_m(\lambda_2) \\ 
														\vdots \\ 
														w^T_n \Phi_m(\lambda_n) \end{bmatrix}.
  \end{eqnarray*}
  The results follows taking the norm $\|\cdot \|_{A^{-1}}$.
  \end{proof}

      Lemma \ref{lem:block_poly_2} shows that the block residual can be expressed as the trace of a 
      sum of $s\times s$ matrices with the matrix polynomial evaluated on each eigenvalue, and the 
      weights $w_i w^T_i$ with $i=1,\ldots,n$ are $s\times s$ 
symmetric positive semidefinite matrices of rank-one.
Note that for $s=1$, Lemma \ref{lem:block_poly_2} reduces to the CG case. Evaluating the
polynomial on each 
    of the eigenvalues and the elimination of the operator `$\circ$'  simplify the development of a superlinear 
      bound in the block case.

As stated in the previous section, $k_1$ 
denotes the number of eigenvalues taken in the lowest part 
      and $k_2$  the number of eigenvalues taken in the upper part of the spectrum to perform our analysis.


For the sake of simplicity, we begin by stating
and proving the following theorem for the superlinear bound 
in the case of $k_1=1$ and $k_2=0$. That is, we consider the eigenpairs corresponding to the 
   lowest part of the spectrum $(\lambda_1,v_1)$ and its corresponding lowest Ritz value at the iteration 
   $m$ denoted $\theta^{(m)}_1$. 
Let $\bar{\bf R}_0$ be the block residual where the eigenvector $v_1$ 
 has been deflated, and $\bar{\bf R}_j$ is the residual of the block CG process (comparison process) starting 
   with  $\bar{\bf R}_0$. We show a bound of the form
$\| {\bf R}_{m+j} \|_{A^{-1}} \leq \alpha_m \| \bar{\bf R}_j\|_{A^{-1}}$, where
the factor $\alpha_m$ depends on the spectral information, as shown below.
In other words, we bound the norm of block residual at the $(m+j)$th iteration by  the
norm of the residual at the $j$th iteration of the comparison process.
Thus, when the Ritz value approximates well the eigenvalue, the slope of the graph of the
residual norm of the comparison process should mimic the slope of the residual
norm we are trying to bound.


   \begin{theorem}\label{thm:superlinear_spectrum_bcg}
     Let ${\bf R}_{m+j}$ be the block CG residual at the $(m+j)$th iteration and $\bar{\bf R}_j$ be 
     the residual after $j$ iterations 
    of block CG applied to $\bar{\bf R}_0=(I-\Pi_Q){\bf R}_m$, with $\Pi_Q=v_1v_1^T$ and 
    ${\bf R}_m=\Phi_m(A)\circ {\bf R}_0$, {\it i.e.}, 
            \begin{equation}\label{eq:block_bound_vds}
                  \bar{\bf R}_j = \Psi_j(A)\circ \bar{\bf R}_0,
            \end{equation}
     \noindent where $\Psi_j(\lambda)\in {\mathbb{G}}_{m,s}$   is the corresponding matrix polynomial 
     of degree $j$ for the new block CG residual.  Then, 
    \begin{equation}\label{block:bound}
        \| {\bf R}_{m+j} \|_{A^{-1}}\leq \alpha_{m, 1, 0} \| \bar{\bf R}_j \|_{A^{-1}}, 
    \end{equation}
     where $\alpha_{m,1,0}=  \frac{\theta^{(m)}_1}{\lambda_1} 
     \max_{\lambda_i\neq \lambda_1} \frac{|\lambda_i-\lambda_1|}{|\lambda_i-\theta^{(m)}_1|}  \cdot$ 
   \end{theorem}
   \begin{proof}
   Let $\Omega_m(\lambda)$ be a matrix-valued polynomial constructed as follows,
   \begin{equation}\label{eq:omega_matrix_polynomial}
            \Omega_m(\lambda)
            =\frac{\theta^{(m)}_1}{\lambda_1}(\lambda -\lambda_1 )(\lambda  
                  - \theta^{(m)}_1 )^{-1}\Phi_m(\lambda) \in  {\mathbb{G}}_{m,s}.
   \end{equation}
By the optimality 
  property (\ref{eq:bcg_pol_21}) of block CG, we have the following relation,
   $$ \left\lVert {\bf R}_{m+j} 
   \right\rVert_{A^{-1}} \leq \left\lVert \Psi_j(A)\circ \left[\Omega_m(A)\circ 
              {\bf R}_0\right]\right\rVert_{A^{-1}}.$$
    Following a procedure similar to that used in Lemma \ref{lem:block_poly_2}, we have
     \begin{eqnarray}\label{eq:thm_bcg_superlinear_1}
         \| {\bf R}_{m+j} \|^2_{A^{-1}}  \leq \left\lVert V \begin{bmatrix} w^T_1 \Omega_m(\lambda_1) \Psi_j(\lambda_1) \\ 
	w^T_2 \Omega_m(\lambda_2) \Psi_j(\lambda_2) \\ 
		\vdots \\ 
  w^T_n \Omega_m(\lambda_n) \Psi_j(\lambda_n) \end{bmatrix} \right\rVert^2_{A^{-1}}.
    \end{eqnarray}
  
   Using expression (\ref{eq:omega_matrix_polynomial}), the right hand side of 
   inequality (\ref{eq:thm_bcg_superlinear_1}) 
is equal to
 \begin{equation}\label{eqn:aux1}
\left\lVert V \begin{bmatrix} w^T_1 \frac{\theta^{(m)}_1}{\lambda_1}(\lambda_1 -\lambda_1 )(\lambda_1  
                  - \theta^{(m)}_1 )^{-1} \Phi_m(\lambda_1) \Psi_j(\lambda_1) \\ 
						w^T_2 \frac{\theta^{(m)}_1}{\lambda_1}(\lambda_2 -\lambda_1 )(\lambda_2  
                  - \theta^{(m)}_1 )^{-1}  \Phi_m(\lambda_2) \Psi_j(\lambda_2) \\     
                   \vdots \\ 
					    w^T_n \frac{\theta^{(m)}_1}{\lambda_1}(\lambda_n -\lambda_1 )(\lambda_n  
          - \theta^{(m)}_1 )^{-1}  \Phi_m(\lambda_n)  \Psi_j(\lambda_n) \end{bmatrix} \right\rVert^2_{A^{-1}} .
\end{equation}

By construction, the first row of the matrix on the right hand side of (\ref{eq:thm_bcg_superlinear_1}),
(or equivalently in  expression (\ref{eqn:aux1})) is zero. It follows then that by  evaluating the matrix 
polynomial $\Omega(\lambda)$ on each eigenvalue and taking into account that 
the matrices $w_i w^T_i$ are positive semidefinitive, we have that (\ref{eqn:aux1}) is equal to
\begin{multline}
\left\lVert V \begin{bmatrix} 0 \\ 
		w^T_2 \frac{\theta^{(m)}_1}{\lambda_1}(\lambda_2 -\lambda_1 )(\lambda_2  
                 - \theta^{(m)}_1 )^{-1}  \Phi_m(\lambda_2) \Psi_j(\lambda_2) \\ 	\vdots \\ 
			    w^T_n \frac{\theta^{(m)}_1}{\lambda_1}(\lambda_n -\lambda_1 )(\lambda_n  
    - \theta^{(m)}_1 )^{-1}  \Phi_m(\lambda_n)  \Psi_j(\lambda_n) \end{bmatrix} \right\rVert^2_{A^{-1}} \\
        =  \left[\frac{\theta^{(m)}_1}{\lambda_1} \right]^2  \trace \left[\sum^n_{i=2} (\lambda_i -\lambda_1 
                  )(\lambda_i  
                  - \theta^{(m)}_1 )^{-1} \frac{1}{\lambda_i}    
            (\Phi_m(\lambda_i) \Psi_j(\lambda_i))^T w_i w^T_i \Phi_m(\lambda_i)\Psi_j(\lambda_i) \right].
 \end{multline}
 In the last expression we take the maximum over the scalar factors, and obtain the following bound 
for (\ref{eq:thm_bcg_superlinear_1}).
\begin{eqnarray*}\label{eq:thm_bcg_superlinear_3}
    \leq \left[ \frac{\theta^{(m)}_1}{\lambda_1} \max_{\lambda_i\neq \lambda_1}   
       \frac{|\lambda_i-\lambda_1|}{|\lambda_i-\theta^{(m)}_1|}\right]^2 \trace\left[\sum^n_{i=2}  \frac{1}{\lambda_i}    
            (\Phi_m(\lambda_i) \Psi_j(\lambda_i))^T w_i w^T_i \Phi_m(\lambda_i)\Psi_j(\lambda_i)\right].
 \end{eqnarray*}
 Using Lemma \ref{lem:block_poly_2} and since in the sum starts  with $i=2$, we obtain the following  
   expression to the block CG residual.  
   \begin{eqnarray}
   \left\lVert  {\bf R}_{m+j} \right\rVert^2_{A^{-1}}  & \leq & \left[ \frac{\theta^{(m)}_1}{\lambda_1} 
   \max_{\lambda_i\neq \lambda_1}   
       \frac{|\lambda_i-\lambda_1|}{|\lambda_i-\theta^{(m)}_1|}\right]^2 \left\lVert \Psi_j(A)\circ \Phi_m(A) \circ 
        (I-\Pi_Q){\bf R}_0 
       \right\rVert^2_{A^{-1}} \nonumber \\
     & = & \left[ \frac{\theta^{(m)}_1}{\lambda_1} \max_{\lambda_i\neq \lambda_1} 
      \frac{|\lambda_i-\lambda_1|}{|\lambda_i-\theta^{(m)}_1|}\right]^2 \left\lVert \bar{\bf R}_j \right\rVert^2_{A^{-1}}.
    \end{eqnarray}
    The result follows by taking the square root. 
  \end{proof}

Observe that the expression in (\ref{eq:omega_matrix_polynomial})  is indeed a matrix polynomial.
The Ritz value $\theta^{(m)}_1$ is a latent root of $\Phi_m$, so that in (\ref{eq:omega_matrix_polynomial})
this root is removed, and $\lambda_1$ is added as a root. 
Furthermore, this auxiliary matrix polynomial converges to the residual polynomial $\Phi_m(\lambda)$
as $m$ increases, since he Ritz value $\theta^{(m)}_1$ converges
    to $\lambda_1$ and 
$ 
(\lambda  - \lambda_1 )(\lambda  - \theta^{(m)}_1)^{-1}   
\approx 
{\cal  I}_s 
. $

We state now the general case, that is, a comparison with a process with multiple eigenvalues
deflated and at both ends of the spectrum. The proof follows in the same manner 
as the proof of Theorem~\ref{thm:superlinear_spectrum_bcg}, by appropiately constructing
the auxiliary matrix polynomial $\Omega_m(\lambda)$. 
      

    \begin{theorem}\label{thm:superlinear_spectrum_bcg_general}
      Let ${\bf R}_{m+j}$ be the block CG residual at the $(m+j)$th step and $\bar{\bf R}_j$ be the residual 
      after $j$ steps of block CG applied to $\bar{\bf R}_0=(I-\Pi_Q){\bf R}_m$, with $\Pi_Q=QQ^T$ 
      and ${\bf R}_m=\Phi_m(A)\circ {\bf R}_0$, where $Q\in \mathbb{R}^{n\times k}$ with 
      $k=k_1+k_2$, and  
            \begin{equation}\label{eq:block_bound_vds_general}
                  \bar{\bf R}_j = \Psi_j(A)\circ \bar{\bf R}_0,
            \end{equation}
       \noindent where $\Psi_j(\lambda)\in \mathbb{G}_{m,s}$ is the corresponding block CG matrix 
       polynomial of degree $j$. Then, 
      \begin{equation}\label{spec:bound}
           \left\lVert  {\bf R}_{m+j} \right\rVert_{A^{-1}} \leq \alpha_{m,k_1,k_2} \left\lVert  \bar{\bf R}_j \right\rVert_{A^{-1}}, 
      \end{equation}
      where $\alpha_{m,k_1,k_2}$ is given by 
       \begin{equation}\label{eq:alpha_mlr1}
                    \alpha_{m,k_1,k_2} = \max_{j_1>k_1,\; j_2\geq k_2} \prod^{k_1}_{j=1} 
                             \frac{\theta_j^{(m)}}{\lambda_j} \left| \frac{\lambda_{j_1}-\lambda_j}{
                    \lambda_{j_1}-\theta^{(i)}_{j}} \right| \prod^{k_2}_{j=1} \frac{\theta^{(m)}_{n+1-j}}{\lambda_{n+1-j}} \left| 
                    \frac{ \lambda_{n+1-j}
                            -\lambda_{n-j_2}}{ 
                    \theta^{(m)}_{m+1-j}-\lambda_{n-j_2}}\right| .
      \end{equation}
   \end{theorem}

\begin{proof}
  Let $\Omega_m(\lambda)$ be a matrix polynomial constructed as follows,
  {\begin{multline}\label{eq:omega_matrix_polynomial:Teo5}
            \Omega_m(\lambda) =  \prod^{k_1}_{j=1} \frac{\theta_j^{(m)}}{\lambda_j} (\lambda_{j_1}  - \lambda_j )(\lambda_{j_1}  - \theta^{(m)}_{j})^{-1} \times \nonumber \\ \prod^{k_2}_{j=1} \frac{\theta^{(m)}_{n+1-j}}{\lambda_{n+1-j}}(\lambda_{n+1-j}   - \lambda_{n-j_2} )( \theta^{(m)}_{m+1-j}  - \lambda_{n - j_2} )^{-1} \Phi_m(\lambda).
   \end{multline}}
   for $j_1 > k_1$ and $j_2 \geq k_2$, and since $\Omega_m(0)= {\cal I}_s$,  $\Omega_m(\eta) \in {\mathbb{G}}_{m,s} $. 
 The result follows using a similar procedure employed in Theorem~\ref{thm:superlinear_spectrum_bcg}. 
 \end{proof}
 

We note that theorems~\ref{thm:superlinear_spectrum_bcg} 
and~\ref{thm:superlinear_spectrum_bcg_general} 
reduce to the results 
obtained in \cite{van1987convergence} when studying superlinear bound in CG,
when we consider the special case $s=1$.


We end this section by remarking that 
Theorem \ref{thm:superlinear_spectrum_bcg_general} holds in particular
when an eigenvalue has multiplicity $\kappa > 1$. In this case, in the expression for
${\alpha}_{m,k_1,k_2}$ some factors are repeated $\kappa$ times.


\section{Numerical Experiments}\label{sec:experiments_bcg}
We present a series of numerical experiments illustrating the quality of 
both the subspace-based bounds \eqref{eq:cota_ss_cg11} and
the spectral-based bounds \eqref{block:bound}, and we compare them.
We discuss different cases, 
namely different eigenvalue distributions of the coefficient matrix $A$,
as well as 
taking different dimensions of the invariant subspaces in the comparison process.
We also consider the effect of taking increasing numbers of right hand sides.

Without loss of generality, we consider diagonal matrices in our first five examples.
In Example~\ref{ex:ex_real_matrix} we consider a preconditioned model problem.
Recall that the coefficient $\alpha_{m,k_1,k_2}$ in \eqref{block:bound}
depends on $k_1$ eigenvalues in the lower part of the spectrum, and $k_2$ eigenvalues
in the upper part of the spectrum. Different cases of $k_1$ and $k_2$ are considered,
both when the eigenvalues are simple, and when they have multiplicity $\kappa > 1$.
Recall also that in the bound \eqref{eq:cota_ss_cg11} we have an invariant subpace
whose basis are the columns of the  matrix $Y$. In our
experiments, these columns are $y_i=Az_i$, $z_i\in \mathbb{K}_m(A, {\bf R}_0)$,
     and $z_i$ are taken to be Ritz vectors.

For each example,
two sets of complementary results are presented. In one set, we follow the evolution
of the Ritz values as they converge to the corresponding eigenvalues, and look
at the behaviour of the two constants used in our bounds,
$\alpha_{m,k_1,k_2}$ in \eqref{block:bound} and $\gamma_m$ in \eqref{eq:cota_ss_cg11}.
In this case, we use the bounds \eqref{block:bound} and \eqref{eq:cota_ss_cg11} for $j=0$
as $m$ increases.
We can observe than when the Ritz values are close enough to the eigenvalues,
or when $\gamma_m$ is small enough, the convergence curve changes slope; though
this point is hard to pinpoit exactly.
The second set of experiments show that the bounds obtained follow closely the
slope of the convergence curves, especially after the change of slope, sometimes
called the onset of superlinearity, has taken hold.

\begin{example}\label{exp:example_1}
{\rm
Our first example is for the case $s=1$, and thus, the block method reduces to the 
standard case, and our bounds reduce to that presented in \cite{van1987convergence}. 
We consider $k_1=1$, $k_2=0$, that is, only one Ritz value (and one Ritz vector), at the
lowest part of the spectrum.
We do so as to examine three situations in which the Ritz value $\theta_1^{(m)}$
is either very close to the first eigenvalue $\lambda_1$ (Case \ref{exp:example_1}.a),
$\theta_1^{(m)}$ is between $\lambda_1$ and $\lambda_2$ (Case \ref{exp:example_1}.b),
or $\theta^{(m)}_1 > \lambda_2$ (Case \ref{exp:example_1}.c),
as was done in~\cite{van1993superlinear}.
We thus study a $100\times 100$ diagonal matrix  with
    eigenvalues $0.1$,  $0.2$, $0.3$, $0.4$, $5$, $\ldots$ , $100$. The right-hand side is a vector  
    with all unit entries, the initial  approximation ${\bf X}_0=x_0$ is the zero vector, and the initial 
    residual is ${\bf R}_0=b$; {\it i.e.}, the residual has equal components in all eigenvectors that form 
    an eigenbasis so that no one eigendirection is favored. 
   
Recall that
in the expressions (\ref{eq:cota_ss_cg11}) and (\ref{spec:bound}), ${\bf R}_{m+j}$ is 
the block CG residual at the $(m+j)$th step, and $\bar{\bf R}_j$ is the residual after $j$ steps of block CG applied to $\bar{\bf R}_0=(I-\Pi_Q){\bf R}_m$. The counter $ m $ specifies the 
status of iteration advance of block CG when the bound starts to be employed. 
Hence, in the first set of experiments, as $m$ increases, $\theta_1$  convergences
towards $\lambda_1$, as shown, e.g., in  Table \ref{tab:t_close_lambda} below, 
and for fixed $m$, as $j$ increases, we look at the behavior of the bounds,
as shown, e.g., in Table \ref{tab:b1b2_prediction}. 
  
   {\bf {Case \ref{exp:example_1}.a}. 
   The first Ritz value $\theta_1$ is very close to $\lambda_1$.} 
Our first set of experiments are reported in Table \ref{tab:t_close_lambda}, 
where we  show  the behavior of $\alpha_{m, 1, 0}$ and $\gamma_m$ at each  iteration $m$. Bounds
    (\ref{eq:cota_ss_cg11})  and  (\ref{block:bound})  are computed with $j=0$ and are 
presented for several values of $m$. 
In other words, the bounds reduced to
    \begin{equation} \label{cota:experimentalj0}
       b_1:=\| \bar{\bf R}_0\|_{A^{-1}} 
      +\gamma_m\|\Pi_Q{\bf R}_m\|_{A^{-1}},  \mbox{ ~and ~ } 
       b_2 := {\alpha}_{m, 1, 0}\| \bar{\bf R}_0 \|_{A^{-1}}.  
      \end{equation}  
   The comparison residual at iteration $m$ is $\bar{\bf R}_0=\|(I-\Pi_Q){\bf R}_m\|_{A^{-1}}$. 
   Observe that  despite of the fact that the convergence of the Ritz vectors is slow and   
   non-monotone~\cite{Saad1980}, which has influence on $\gamma _m$, the subspace bound $b_1$ gives 
   sharper estimates than the spectral bound $b_2$; see Table \ref{tab:t_close_lambda}.
   
 \begin{table}[htb]
   \centering
   \caption{Case \ref{exp:example_1}.a. The first Ritz value $\theta^{(m)}_1$ is very 
   close to smallest eigenvalue $\lambda_1$. Parameters $k_1=1$, $k_2=0$, $s=1$ and $j=0$. Factors $\gamma _m$ and ${\alpha}_{m, 1, 0}$
     are computed using (\ref{eq:gamma_identity1}) and (\ref{block:bound}), respectively. The subspace bound $b_1$
 and the spectral bound $b_2$ correspond to 
   (\ref{cota:experimentalj0}).  \label{tab:t_close_lambda} }
   \resizebox{\textwidth}{!}{
  {
 \begin{tabular}{c|cc|cc|ccc|c}
  $m$  & $\theta_1$ & $\theta_1/\lambda_1$ &  $\gamma _m$ & ${\alpha}_{m, 1, 0}$   & $b_1$ & $b_2$ 
       & $\|\bar{\bf R}_0\|_{A^{-1}} $ & $\| {\bf R}_m \|_{A^{-1}} $ \\ \hline
 $30$ & 0.12659    & 1.26597                       & 0.61127 & 1.72469  &0.77367 & 0.92530    & 0.53650 & 0.66210     \\
 $31$ & 0.12280    & 1.22800                       & 0.56286  & 1.59067 &  0.67380 & 0.79217   & 0.49801 & 0.58784     \\
 $32$ & 0.11708    & 1.17082                       & 0.48434  &1.41203  & 0.53402  & 0.60523  & 0.42862 & 0.48069     \\
 $33$ & 0.11138    & 1.11386                       & 0.39936  &1.25698 & 0.39652  & 0.42922  & 0.34147 & 0.36825     \\
 $34$ & 0.10771    & 1.07712                       & 0.33933 & 1.16715    &  0.29890 & 0.31279  & 0.26799 & 0.28305     \\  
 \end{tabular}}}
 \end{table}

 Observe that as $m$ increases, then  $ {\alpha}_{m, 1, 0} $ decreases
     (ideally it tends to one) and $\gamma _m$ also decreases (ideally tends to zero).  
    Recall that  when $ {\alpha}_{m, 1, 0} = 1$ and $\gamma_m = 0$, then both bounds 
    coincide. However,  also note that the spectral  bound $b_2$ is more sensitive  
    to changes $ {\alpha}_{m, 1, 0}$ than the  bound $b_1$ to changes in $\gamma _m$. This means 
    that $\gamma _m$ does not need to be close to zero in order for the subspace bound $b_2$ obtain a better 
    approximation than the spectral bound~$b_1$.

  \begin{table}[htb]
 \centering
 \caption{Case \ref{exp:example_1}.a. The first Ritz value $\theta_1$ very close to smallest eigenvalue $\lambda_1$. Parameters are
 $k_1=1$, $k_2=0$, $s=1$, and iteration $m=33$. Factors ${\alpha}_{m, k_1, k_2} = {\alpha}_{33, 1, 0} $  and  
 $\gamma_{m}=\gamma_{33}$ are computed using (\ref{block:bound}) and (\ref{eq:gamma_identity1}), respectively. 
The subspace bound $b_{1,j}$ and  the spectral bound $b_{2,j}$ are given by (\ref{cota:experimentalJ}).}
 \label{tab:b1b2_prediction}
 {\small
 \begin{tabular}{ccccc}
                          & $\theta_1$ & ${\alpha}_{33, 1, 0}$   & $\gamma_{33}$  & $\| {\bf R}_{33}\|_{A^{-1}}$ \\ \cline{2-5} 
                          & 0.11138    & 1.25698 & 0.39936                      & 0.36825         \\ \hline \hline 
\multicolumn{1}{c|}{$j$}  & $b_{1,j}$      & $b_{2,j}$   &  \multicolumn{1}{c|}{\footnotesize $\|\bar {\bf R}_{j} \|_{A^{-1}}$ } & 
{\footnotesize  $\| {\bf R}_{33+j} \|_{A^{-1}}$} \\ \hline
\multicolumn{1}{c|}{$1$}  & 0.34058 & 0.35903    & \multicolumn{1}{c|}{0.28563} & 0.28305         \\
\multicolumn{1}{c|}{$2$}  & 0.29627 & 0.30365    & \multicolumn{1}{c|}{0.24157} & 0.23350         \\
\multicolumn{1}{c|}{$3$}  &  0.27034 & 0.27142    & \multicolumn{1}{c|}{0.21594} & 0.20939         \\
\multicolumn{1}{c|}{$4$}  &0.25733  &  0.25542    & \multicolumn{1}{c|}{0.20321} & 0.19836         \\
\multicolumn{1}{c|}{$5$}  & 0.25064  & 0.24739     & \multicolumn{1}{c|}{0.19684} & 0.19279         \\
\multicolumn{1}{c|}{$6$}  & 0.24635  &  0.24246    & \multicolumn{1}{c|}{0.19293} & 0.18896         \\
\multicolumn{1}{c|}{$7$}  & 0.24228 &  0.23803    & \multicolumn{1}{c|}{0.18944} & 0.18476         \\
\multicolumn{1}{c|}{$8$}  &0.23651  &  0.2320      & \multicolumn{1}{c|}{0.18474} & 0.17775         \\
\multicolumn{1}{c|}{$9$}  & 0.22602  & 0.22135    & \multicolumn{1}{c|}{0.17654} & 0.16362         \\
\multicolumn{1}{c|}{$10$} &  0.20638 & 0.20139    & \multicolumn{1}{c|}{0.16165} & 0.13714        
\end{tabular}}
\end{table}
 
In the second set of experiments, we examine our two bounds in the 
   same stage of convergence of the Ritz values to their corresponding eigenvalues. 
   To this end, we fix $m=33$ and take iterations
   $j$ from $j=1$ to $j=10$. 
For easier reading we denote our bounds as follows,
   \begin{equation}\label{cota:experimentalJ}
   \left \{ 
   \begin{array}{ccl}
      b_{1,j} &=& \min_{ {\bf D} \in A\mathbb{K}_j(A, {\bf R}_m)} 
                         \left\{\|(I-\Pi_Q)({\bf R}_m - {\bf D})\|_{A^{-1}}  \right. \vspace{0.07in} \\ 
          & & \qquad \qquad \qquad \qquad \qquad \left.  +\gamma _m \|\Pi_Q({\bf R}_m - {\bf D})\|_{A^{-1}}  \right\}
  \qquad \mbox{ and } \vspace{0.1in} \\  
     b_{2,j}&=&\alpha_{m, k_1, k_2}\|\bar{\bf R}_j\|_{A^{-1}}.
   \end{array}
   \right .
   \end{equation}

  Table \ref{tab:b1b2_prediction} shows the behavior of expressions the subspace bound $b_{1,j}$ 
  and the spectral bound $b_{2,j}$ 
Note that for this example, both bounds behave similarly.
 

 {\bf Case \ref{exp:example_1}.b.  The Ritz value $\theta_1$ is between $\lambda_1$ and $\lambda_2$.} 
The first set of experiments, for $m=21,\ldots,25$, are presented in 
Table \ref{tab:middle}. 
In this interval, the factor 
  ${\alpha}_{m, 1, 0}$ varies considerable with $m$. This is in contrast with $\gamma _m$ which
does not vary as much.. 
The variation of ${\alpha}_{m, 1, 0}$ is due to the proximity of $\theta^{(m)}_1$ to $\lambda_2$, 
which yields a very rough estimate by the spectral bound $b_2$ of the block CG behavior. 
The estimate is improved either once $\theta^{(m)}_1$  moves away from $ \lambda_2$ 
or by replacing the factor $\alpha_m$ by $\theta^{(m)}_1/\lambda_1$ as suggested in 
\cite{vanderSluis1986}. The latter  
results in a bound ($\frac{\theta_1}{\lambda_1}\|\bar{\bf R}_0\|_{A^{-1}}$)  which is
sharper than either of our bounds in this case.
   \begin{table}[htb]
   \centering
  \caption{{Case \ref{exp:example_1}.b.}  Ritz value $\theta_1$ in $\left[\lambda_1,\lambda_2\right]$. 
  Parameters  $k_1=1$, $k_2=0$, $s=1$, and $j=0$. Factors  ${\alpha}_{m, 1, 0}$ and $\gamma _m$ are computed using  (\ref{block:bound}) and  (\ref{eq:gamma_identity1}), respectively. 
The subspace bound $b_1$ and the spectral bound $b_2$ correspond to (\ref{cota:experimentalj0}). 
}
  \label{tab:middle}
  \resizebox{\textwidth}{!}{
  \begin{tabular}{c|cc|cc|ccl|c}
  $m$ & $\theta^{(m)}_1$ & $\theta_1/\lambda_1$ &  $\gamma_m$ & ${\alpha}_{m, 1, 0}$ & $b_1$   & $b_2$   
         & $\|\bar{\bf R}_0\|_{A^{-1}}$  & $\| {\bf R}_m \| _{A^{-1}}$ \\ \hline 
  21  & 0.20181    & 2.0181               & 0.86898  & 110.93237   & 2.11922  & 113.62360  & 1.02426  & 1.62383     \\
  22  & 0.17786    & 1.7786               & 0.80929  & 8.0359        & 1.78204  & 6.28417     & 0.964314 & 1.39672     \\
  23  & 0.15595    & 1.5595               & 0.74794  & 3.5404        & 1.43954  & 2.47657     &  0.85200  & 1.15887     \\
  24  & 0.14271    & 1.4271               & 0.70553  & 2.4913        & 1.18655  & 1.59363     & 0.73686  & 0.97427     \\
  25  & 0.13622    & 1.3622               & 0.68186  & 2.1359        & 1.03745  &  1.29275    & 0.65704  & 0.86194     \\ 
  \end{tabular}}
 \end{table}

  \begin{table}[h!]
 \centering
 \caption{{Case \ref{exp:example_1}.b.}  Residual bounds at $m=23$, when $\theta_1$ is in the middle of $ \left[\lambda_1,\lambda_2\right]$. Parameters  $k_1=1$, $k_2=0$ and $s=1$. 
   Factors  ${\alpha}_{23, 1, 0}$ corresponds to (\ref{block:bound}) 
    and $\gamma _m$ to  (\ref{eq:gamma_identity1}).  
The subspace bound $b_{1,j}$ and the spectral bound  $b_{2,j}$ correspond to~(\ref{cota:experimentalJ}). 
   }\label{tab:b1b2_middle}
 {\small
 \begin{tabular}{ccccc}
                          & $\theta_1$ & $\alpha_{23, 1, 0}$   & $\gamma_{23}$                    & $\|{\bf R}_{23}\|_{A^{-1}}$           \\ \cline{2-5} 
                          & 0.15595    & 3.54044 & 0.74794                      & 1.15887         \\ \hline  \hline
 \multicolumn{1}{c|}{$j$}  & $b_{1,j}$      &   $b_{2,j}$    & {\footnotesize $\|\bar {\bf R}_{j} \|_{A^{-1}}$} & {\footnotesize  $\| {\bf R}_{23+j} \|_{A^{-1}}$ }\\ \hline
 \multicolumn{1}{c|}{$1$}   &  1.28592 &  2.47657   & \multicolumn{1}{c|}{0.69950} & 0.97427         \\
 \multicolumn{1}{c|}{$2$}   & 1.15295  & 2.01543    & \multicolumn{1}{c|}{0.56925} & 0.86194         \\
 \multicolumn{1}{c|}{$3$}   & 1.06385  & 1.71144    & \multicolumn{1}{c|}{0.48339} & 0.80285         \\
 \multicolumn{1}{c|}{$4$}   & 1.01284  &  1.54247   & \multicolumn{1}{c|}{0.43567} & 0.77255         \\
 \multicolumn{1}{c|}{$5$}   &  0.98455 &  1.45406   & \multicolumn{1}{c|}{0.41070} & 0.75507         \\
 \multicolumn{1}{c|}{$6$}   & 0.96687  & 1.40444    & \multicolumn{1}{c|}{0.39668} & 0.74175         \\
 \multicolumn{1}{c|}{$7$}   & 0.95238  & 1.36986    & \multicolumn{1}{c|}{0.38691} & 0.72694         \\
 \multicolumn{1}{c|}{$8$}   & 0.93562  & 1.33628    & \multicolumn{1}{c|}{0.37743} & 0.70399         \\
 \multicolumn{1}{c|}{$9$}   & 0.90988  & 1.29108    & \multicolumn{1}{c|}{0.36466} & 0.66210         \\
 \multicolumn{1}{c|}{$10$} & 0.86496  & 1.21714    & \multicolumn{1}{c|}{0.34378} & 0.58784        
 \end{tabular}}
 \end{table}

\begin{figure}[h!]
  \centering
  \includegraphics[width=0.55\textwidth]{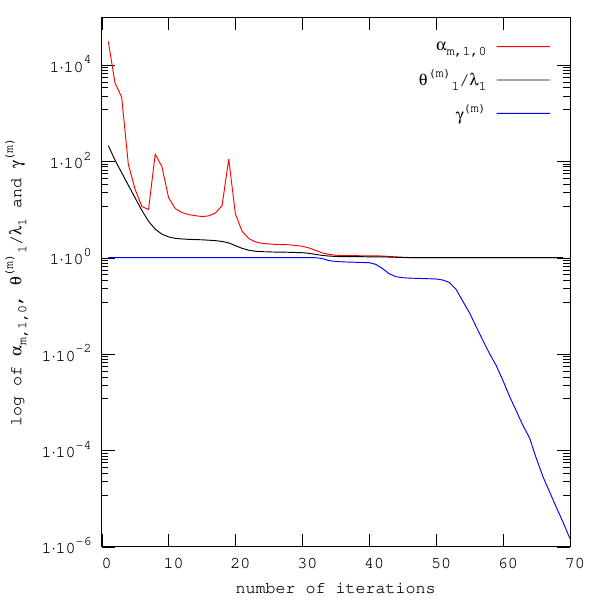}
 \caption{Example \ref{exp:example_1}.c. Behavior of the factors ${\alpha}_{m, 1,0}$ and $\gamma_m$ of our bounds
given in (\ref{cota:experimentalj0}) for each $m$, the number of iterations. The 
smallest Ritz value $\theta^{(m)}_1$ at the $m$th iteration and  
smallest eigenvalue $\lambda_1$ of the matrix $A$.  
 \label{fig:factor_vds_ss_1} }
 \end{figure}

For the second set of experiments, we consider 
$m=23$, when $\theta^{(m)}_1$ fall  almost exactly in the middle of $\left[\lambda_1, \lambda_2 \right]$. Hence, 
  the factor ${\alpha}_{m, 1, 0}$  is large and the spectral bound $b_{2,j}$ overestimates the
residual norm, as shown in 
  Table \ref{tab:b1b2_middle}.
Note that for this case, the subspace bound
   $b_{1, j}$ is sharper than the spectral bound $b_{2, j}$ for all $j$.

   {\bf Case \ref{exp:example_1}.c. When $\theta^{(m)}_1$ is outside 
the interval $[\lambda_1, \lambda_2]$.}
This occurs when $m\leq 20$. The factors 
   ${\alpha}_{m, 1,0}$ and $\gamma_m$  have very different behavior as reported in
   Figure \ref{fig:factor_vds_ss_1}.  Observe that~${\alpha}_{m, 1,0}$ 
has large variations when $\theta^{(m)}_1$ approaches a wrong eigenvalue, and 
   consequently the spectral bound computed with these factors do not provide any useful information. On the other hand, 
   for $m \leq 20$ the value of $\gamma_m$ remains almost equal to $1$, that means that the 
subspace bound 
$b_{1,j}$ 
   in (\ref{cota:experimentalJ}) can 
   be computed as 
   $ \| (I-\Pi_Q) \bar{\bf R}_j \|_{A^{-1}} + \|\Pi_Q  \bar{\bf R}_j  \|_{A^{-1}}$. 
  
In summary, Example \ref{exp:example_1} illustrates the behavior of the bounds in several stages of the 
convergence of the Ritz value $\theta ^{(m)}_1$ towards the eigenvalue $\lambda _1$. 
In {\bf Case \ref{exp:example_1}.a}, both bounds behave similarly
when the Ritz value has sufficiently converged to its corresponding eigenvalue. 
In the other two cases,
the subspace bound is better than the spectral bound when the Ritz value $\theta ^{(m)}_1$ 
is far away of the corresponding eigenvalue $\lambda _1$.
}
 \end{example}

  \begin{example}\label{exp:example_2}
{\rm 
Here, we consider the four smallest eigenvalues in the construction of the bounds, i.e., $k_1=4$ and $k_2=0$.
    In this experiment, the matrix $A$ and the initial iterate ${\bf X}_0$ are the same as in 
    Example \ref{exp:example_1}. 

In the first set of experiments, we present in 
Table \ref{tab:exp2}, the convergence of the Ritz values $\theta _i^{(m)}$, $i=1, \ldots , 4$,  
to the four lowest eigenvalues, and the behavior of ${\alpha}_{m, k_1, k_2} = {\alpha}_{m, 4, 0}$ 
and $\gamma _m$. 
In addition, we present the subspace bound $b_{1}$ and 
the spectral bound $b_{2}$ given in (\ref{cota:experimentalj0}),  
and the norms of
the comparison residual $\bar{\bf R}_0$ and the block CG residual ${\bf R}_m$. 
Note that since  $\theta^{(m)}_4$ is close to an eigenvalue 
different than $\lambda _4$, namely $\lambda_5$, the factor ${\alpha}_{m, 4, 0}$ varies widely.
   
    
    \begin{table}[ht]
  \centering
   \caption{Example \ref{exp:example_2}. Ritz values $\theta_i^{(m)}$, $i=1, \cdots, 4$. The 
   bounds area computed with $j=0$, and parameters  $k_1 = 4$ and $k_2 = 0$. 
   Factors  ${\alpha}_{m, 4, 0}$ corresponds to (\ref{eq:alpha_mlr1}) 
   and $\gamma _m$ to (\ref{eq:gamma_identity1}). 
The subspace bound $b_1$ and the spectral bound $b_2$ correspond to (\ref{cota:experimentalj0}). 
The four smallest eigenvalues are $\lambda_1=0.1,\lambda=0.2,\lambda=0.3$ and $\lambda=0.4$.
   \label{tab:exp2} }
  \resizebox{\textwidth}{!}{
  \begin{tabular}{c|cccc|ccccc|lc}
   $m$      & $\theta^{(m)}_1$ & $\theta^{(m)}_2$ & $\theta^{(m)}_3$ & $\theta^{(m)}_4$ & $\alpha_{m,4,0}$     
   & $\gamma_m$ & $b_1$ & $b_2$ & $\|\bar{\bf R}_0\|_{A^{-1}}$ & $\| {\bf R}_{m}\|_{A^{-1}}$ \\ 
   \hline
  39       & 0.10485          & 0.24301    & 0.39061          & 5.00336      & 29249.6 & 0.99998 &   0.23883   &1340.984 & 0.04584 & 0.19836     \\ 
  40      & 0.10469          & 0.24235    & 0.39033          & 4.99812      & 52236.3 & 0.99996 &   0.22012   & 1547.637 &  0.02962 & 0.19279 \\ 
  41       & 0.10458          & 0.24189    & 0.39013         & 4.98636       & 7124.5   & 0.99983 &   0.21305   &   184.532 &   0.02590 & 0.18896  \\
  42       & 0.10446          & 0.24138    & 0.38991         & 4.35890       & 131.885 & 0.98594 &   0.21336   &  4.53408  &   0.03437 & 0.18476  \\
  43       & 0.10426          & 0.24049    & 0.38949         & 2.33890       & 16.9294 & 0.95700  &    0.21476  &  0.88256  &   0.05213 & 0.17775 \\
  44       & 0.10384          & 0.23842    & 0.38838          & 1.10663      & 5.3865   & 0.88939 &    0.20429  &  0.40350  &   0.07491 & 0.16362 \\
  45       & 0.10302          & 0.23329    & 0.38414          & 0.58421      & 2.40210 & 0.75256 &    0.16799   & 0.16799  &   0.09033 & 0.13714  \\
   46      & 0.10185          & 0.22274    & 0.35799          & 0.42276      & 1.46327 & 0.55404 &     0.11397  & 0.12184   &    0.08327 & 0.10002 \\ 
  \end{tabular}}
  \end{table}

\begin{figure}[htb]
  \centering
  \includegraphics[width=0.6\textwidth,angle=0]{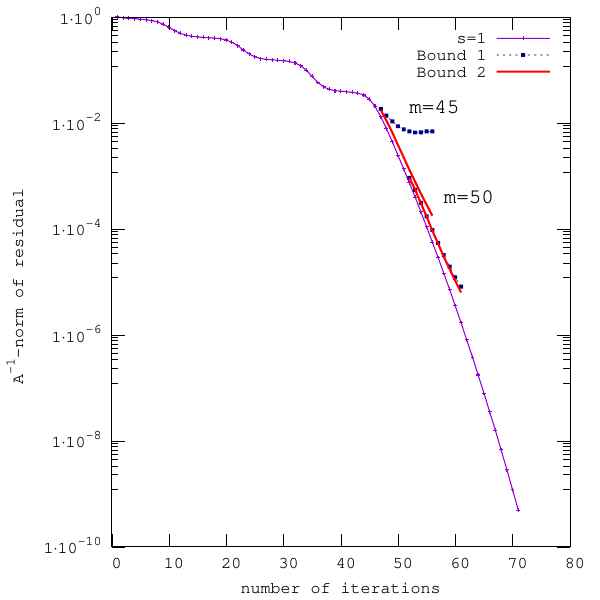}
  \caption{Example \ref{exp:example_2}. 
Block CG residual $\| {\bf R}_m\|_{A^{-1}}$, subspace bound 1 and spectral bound 2 from
  (\ref{cota:experimentalJ}) computed with $m=45$ and $m=50$. Parameters $k_1=4$ and $k_2=0$.
  \label{fig:factor_vds_ss_2}   }
  \end{figure}

The second set of experiments are presented in
Figure \ref{fig:factor_vds_ss_2}, where we see the behavior of 
the residual norm $\| {\bf R}_{m}\|_{A^{-1}}$ of the block CG, the subspace bound $b_{1,j}$ and 
spectral bound~$b_{2,j}$ as expressed in (\ref{cota:experimentalJ}).
%
Note that
at iteration $m=45$ the spectral bound $b_{2,j}$ is sharper than the subspace bound $b_{1,j}$. 
This is because the subspace  bound $b_{1,j}$ with a moderate $\gamma_m$  
 ($\gamma _m=0.75256$) amplifies significantly 
     the residual component on the selected eigenspace~$\mathbb{R}(Q)$. 
At iteration $m=50$ the selected Ritz values are very close to its respective eigenvalues and 
the resulting bounds are almost the same.
}
    \end{example}

\begin{figure}[htb]
    \centering
     \includegraphics[width=0.60\textwidth]{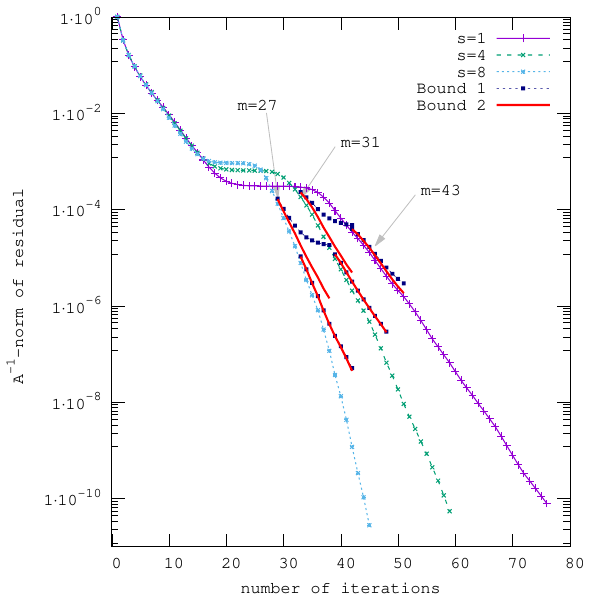}
    \caption{Example \ref{ex:ex_block_1}. Residual convergence behavior for the block 
    CG $\left\lVert  {\bf R}_{m+j} \right\rVert_{A^{-1}} $ for $s=1, 4, 8$,
    subspace and spectral bounds  $b_{1,j}$ and $b_{2,j}$ correspond to expressions  
    (\ref{cota:experimentalJ}).  Parameters $k_1=1$ and $k_2=0$.  Note: $m=43$ correspond to ($\alpha=1.0254$, $\gamma=0.1572$), $m=31$ correspond to ($\alpha=1.201$, $\gamma=0.4083$),   and $m=27$ correspond to ($\alpha=1.0758$, $\gamma=0.2648$).
    \label{fig:residuals_vds_ss_1_block}    }
    \end{figure} 

 \begin{example}\label{ex:ex_block_1}
{\rm
  This example is designed to analyze the effect of the block size $s$ on the 
bounds and on the convergence of the block CG in the presence of 
a single eigenvalue near the origin, i.e., $k_1=1$ and $k_2=0$.
 This experiment considers a diagonal matrix of dimension $404 \times 404$  with eigenvalues 
 on the diagonal with values $0.0005, 0.08,\ldots, 2.42$,  that is, the matrix has one isolated eigenvalue 
 near to zero and the other $403$ eigenvalues are equally distributed between $0.08$ and $2.42$. It is 
 expected to observe superlinear convergence once the method captures 
 the lowest eigenvalue $0.005$.

   Figure \ref{fig:residuals_vds_ss_1_block} shows the convergence history for block sizes $s=1$, $4$ and $8$.  
In this example, and in those in the rest of the section, when $s>1$, the matrix ${\mathbf B}$ has repeated copies of $b$,
but the initial set of vectors in ${\mathbf X}_0$ are nonzero and randomly generated, implying that ${\mathbf R}_0$ has
all distinct columns.
On each residual history, we plot the subspace bound  $b_{1,j}$ and
the spectral bound $b_{2,j}$ at two different stages of convergence: soon after the onset
of superlinearity occurs, and a numbero of iterations later. 
Observe that for values of $m$ relatively large, both bounds capture well the slope 
   of the superlinear regime.  In addition, observe that as the block size is increased, 
block CG hastens the onset of 
superlinear convergence, but going from $s=4$ to $s=8$ the difference in the onset is moderate.      

    \begin{figure}[htb]
   \begin{subfigure}{0.48\textwidth}
  \centering
  \includegraphics[width=0.96\textwidth,angle=-90]{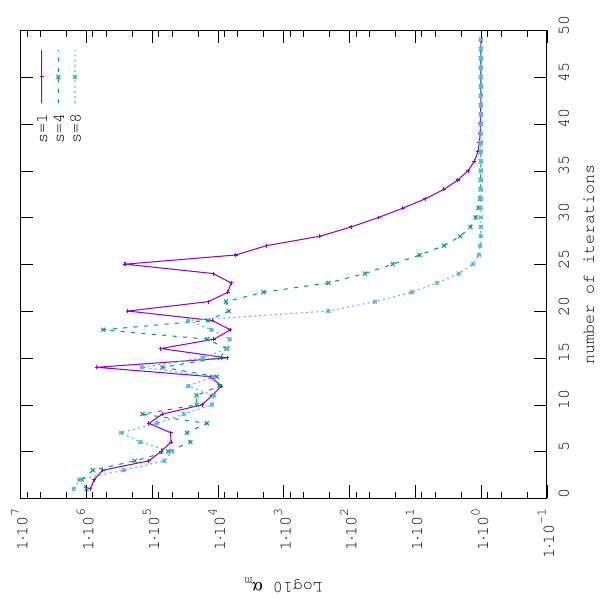}
  \caption{Behavior of $\alpha_{m, 1, 0}$ } \label{fig:m1_fm_block}
  \end{subfigure}
  \begin{subfigure}{0.48\textwidth}
  \includegraphics[width=0.98\textwidth,angle=-90]{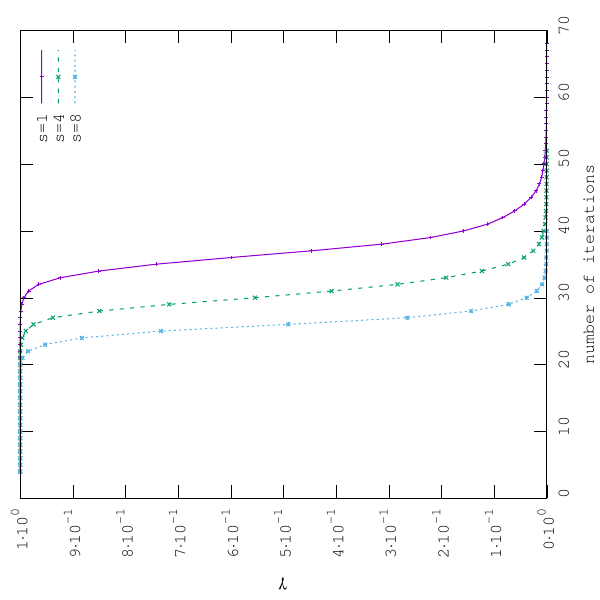}  
  \caption{ Behavior of $\gamma_m$ }   \label{fig:m1_gamma_block}
   \end{subfigure}
   \caption{Example \ref{ex:ex_block_1}.
    Behavior of $\alpha_{m, 1, 0}$ and $\gamma_m$ correspond to the number 
    of iterations $m$. $\alpha_{m, 1, 0}$ is computed with expression (\ref{eq:alpha_mlr1})
     and $\gamma_m$ is computed with  expression (\ref{eq:gamma_identity1}). 
    Parameters $k_1=1$, $k_2=0$ and $s=1, 4, 8$. }
    \label{fig:factor_vds_ss_1_block}
  \end{figure}

Observe also that near the onset of superlinearity (for earlier $m$) 
the spectral bound $b_{2,j}$ is sharper than subspace bound $b_{1,j}$.
This is due to the fact that factor $\gamma_m$ has moderate 
 values ($\gamma_m \in [ 0.5, 1]$)), introducing in the subspace bound $b_{1,j}$ a contribution from the invariant 
   subspace  $\mathbb{R}(Q)$. It is important to remark that moderate values of $\gamma_m$ are in
connection with the rate of convergence of Ritz vectors, which is smaller than that of the convergence of the 
Ritz
   values. This can be appreciated by comparing Figures~\ref{fig:factor_vds_ss_1_block}(a) and \ref{fig:factor_vds_ss_1_block}(b), that the convergence of  $\alpha_{m, 1, 0}$ (to one) is faster than the convergence of  $\gamma_m$ (to zero).

   In addition, in Figure~\ref{fig:factor_vds_ss_1_block}(a), it can be observed the erratic convergence behavior of  
    $\alpha_{m, 1, 0}$. In fact, it takes large oscillatory values before approaching
    the final interval of convergence of the Ritz values. On the other hand, the behavior of $\gamma_m$ is
    smooth and well behaved; see Figure~\ref{fig:factor_vds_ss_1_block}(b). 
    This indicates that the subspace bound $b_{1,j}$  is less
    sensitive to the stage of convergence of the Ritz value, hence it can be used safely to describe the 
    behavior in this region.
}
\end{example}

\begin{example}\label{ex:ex_block_2}
{\rm
  The objective of this example is to analyze the effect of the block size $s$ on the convergence of block CG and the behavior of the bounds to capture the superlinearity in presence of a cluster of eigenvalues near the origin.  To this end, we consider a diagonal matrix of dimension $404 \times 404$  
  with eigenvalues on the diagonal with values
  $0.0005$, $0.0015$, $0.0025$, $0.0035$, $0.0045$, $0.0055$, $0.08,\ldots, 2.42 $. The matrix has six clustered eigenvalues near zero and the rest distributed uniformly between 
  $0.08$ and~$2.42$.


  \begin{figure}[h!]
     \centering
      \includegraphics[width=0.60\textwidth]{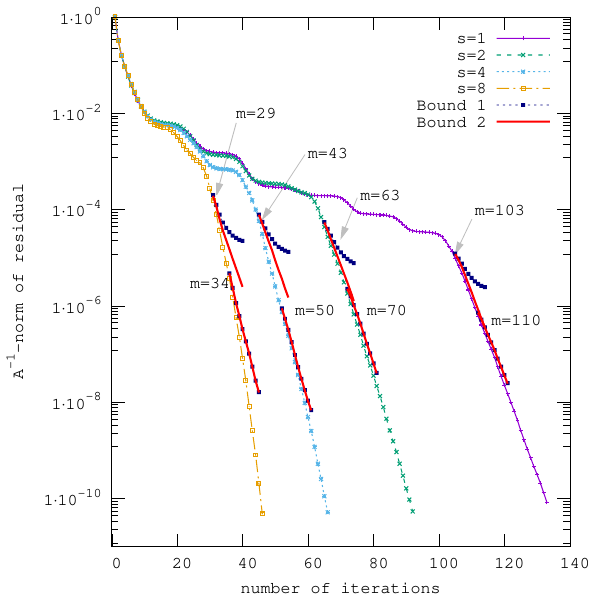}
     \caption{Example \ref{ex:ex_block_2}. Residual convergence behavior 
     for the block CG $\left\lVert  {\bf R}_{m+j} \right\rVert_{A^{-1}} $ for $s=1, 2, 4, 8$,
The subspace bound $b_{1,j}$ and  the spectral bound $b_{2,j}$ are given by (\ref{cota:experimentalJ}).
The parameters are $k_1=6$ and $k_2=0$.  Note: $m=29$ correspond to ($\alpha=1.080$,  $\gamma=0.2702$), $m=43$ correspond to ($\alpha=1.1630$, $\gamma=0.3674$),   $m=63$ correspond to ($\alpha=1.1072$, $\gamma=0.3097$) and $m=103$ correspond to ($\alpha=1.2057$, $\gamma=0.4102$). }
   \label{fig:residuals_vds_ss_2_block}
   \end{figure}

    Figure \ref{fig:residuals_vds_ss_2_block} shows the convergence history of block CG, considering the invariant subspace associated to the six lowest eigenvalues, i.e., both bounds 
    are computed using $k_1=6$ and $k_2=0$ for each block size ($s=1$, $2$, $4$ and $8$). 
Our observations here are very similar to those in the previous example, where a single
eigenvalue was considered.
In particular, note that
that when the block size is increased, the superlinear behavior starts earlier. 
This shows the dependency of the onset of the superlinear behavior on the block size. Moreover, 
as in the previous example,
this example suggests that block CG (with $s> 1$) hastens the onset of the superlinearity in 
the presence of clustered eigenvalues.  
}
\end{example}

\begin{figure}[htb]
  \centering
   \begin{subfigure}{0.50\textwidth}
  \centering
   \includegraphics[width=0.9\textwidth]{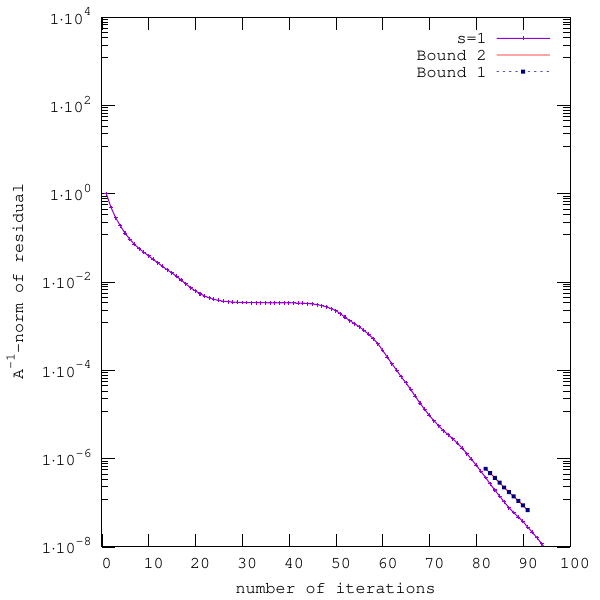}
  \caption{$k_1=1$, $k_2=0$ and $s=1$.} 
  \label{fig:8a}
  \end{subfigure}
   \begin{subfigure}{0.49\textwidth}
  \centering
   \includegraphics[width=0.92\textwidth]{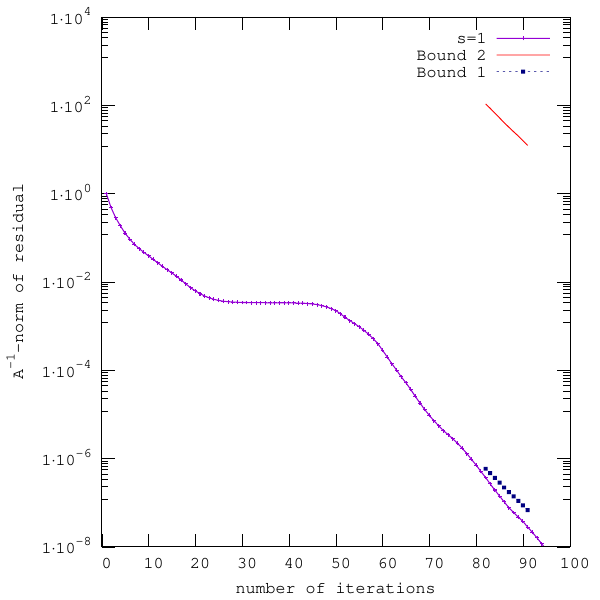}
   \caption{$k_1=2$ , $k_2=0$ and $s = 1$.} 
   \label{fig:8b}
  \end{subfigure}
     \begin{subfigure}{0.49\textwidth}
  \centering
    \includegraphics[width=0.95\textwidth]{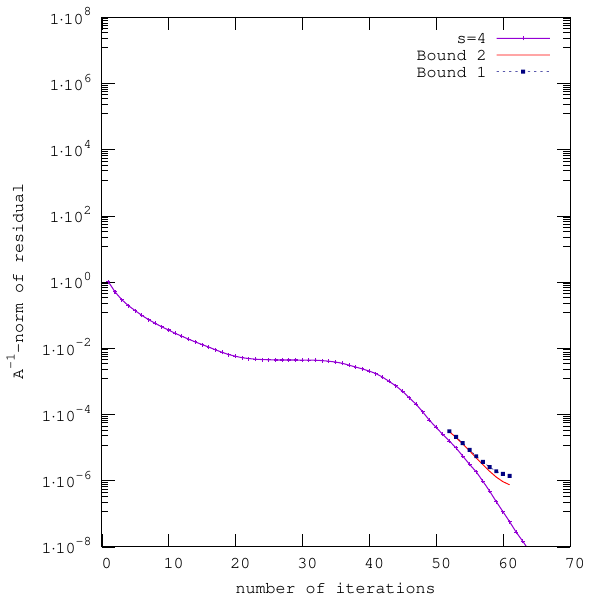}
  \caption{ $k_1=4$, $k_2=0$ and $s=4$.} 
  \label{fig:8c}
  \end{subfigure}
   \begin{subfigure}{0.49\textwidth}
  \centering
 \includegraphics[width=0.93\textwidth]{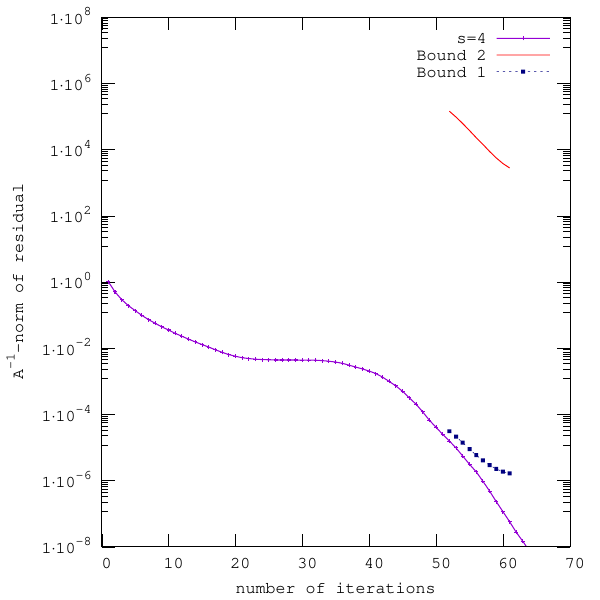}
   \caption{ $k_1=5$, $k_2=0$ and $s=4$.} 
   \label{fig:8d}
  \end{subfigure}
   \caption{Example  \ref{ex:ex_eigen_rep}.  Residual convergence behavior 
     for block CG. 
The subspace bound $b_{1,j}$ and  the spectral bound $b_{2,j}$ are given by (\ref{cota:experimentalJ}).
   (a) Parameters values $s=k_1$ and  $\alpha_{80,1,0}=1.00002$ and $s=1$, 
   (b) Parameters values $k_1=2$ and $\alpha_{80,2,0}=1.83221\times 10^{8}$ and $s=1$,,
   (c) Parameters values $s=k_1$ and  $\alpha_{50,4,0}=1.00108$ and $s=4$,  and
   (d) Parameters values $k_1=5$ and $\alpha_{50,5,0}=4.718055\times 10^{9}$ and  $s=4$.
     \label{fig:ex_4.5} }
 \end{figure}


\begin{example}\label{ex:ex_eigen_rep}
{\rm
In this example we analyze effect of an eigenvalue with algebraic multiplicity $\kappa > 1$ on the 
two bounds we study. To this end, we consider a $384\times 384$ matrix with an eigenvalue $\lambda = 0.0005$ with  
algebraic multiplicity $5$ in the lowest  part  of its  spectrum. 
The rest of eigenvalues are uniformly distributed between  $0.065$ and $5.42$. 
 
 Figure \ref{fig:ex_4.5}  shows the block CG residual, the subspace bound $b_{1,j}$, and  the residual bound~$b_{2,j}$. 
 When the parameter $k_1 =s$, the block size  (see Figures \ref{fig:ex_4.5}(a)  and \ref{fig:ex_4.5}(c)),
 the spectral bound $b_{2,j}$ (defined in (\ref{cota:experimentalJ}) with $\alpha_{m, k_1, k_2}$ 
approximates  adequately the behavior of the residual. On the other hand, if $ k_1  > s$ 
  (see Figures  \ref{fig:ex_4.5}(b)  and \ref{fig:ex_4.5}(d)), then the bound captures the slope but it is 
far from  sharp. In all cases the subspace bound $b_{1,j}$ approximates  adequately the behavior of the residual. 
  
  At this point, it is important to remark that the CG polynomial only captures one copy of the eigenvalue with multiplicity, but block CG can find up to $s$ copies in the case of repeated eigenvalues \cite{underwood}. This remark also applies for clustered eigenvalues. Hence, if $s =  k_1 \leq \kappa $, block CG captures $s$ eigenvalues and the spectral bound $b_{2,j}$ which uses these $k_1=s$ eigenvalues approximates well the residual. However, if  $s <  k_1 \leq \kappa $, then the bound is expected to capture more eigenvalues than the block CG is able to capture, and consequently the approximation bound is not sharp. 
Note the different horizontal scales in Figures \ref{fig:ex_4.5}(c) and (d) for $s=4$, compared to 
that in Figures~\ref{fig:ex_4.5}(a) and (b) for $s=1$.

 The analysis of this example suggest that when an eigenvalue with algebraic multiplicity $\kappa$ is considered, the spectral-based bound approximates adequately the residual behavior when 
the number of eigenvalues taken as reference (denoted by $k_1 \leq \kappa$) is at most equal to the block size $s$.

In this example, it can also be observed that as the block size $s$ is increased, 
the number of eigenvalues captured by block CG is larger, hence the onset of the 
superlinear convergence occurs earlier (compare for instance Figures \ref{fig:ex_4.5}(b)  and \ref{fig:ex_4.5}(d). 
This is in accordance with the observation made in Example \ref{ex:ex_block_2} showing 
a close relationship between block size $s$ and the onset of superlinear convergence.
}
\end{example}

\begin{figure}[htb]
  \centering
   \begin{subfigure}{0.48\textwidth}
    \centering
     \includegraphics[width=0.9\textwidth]{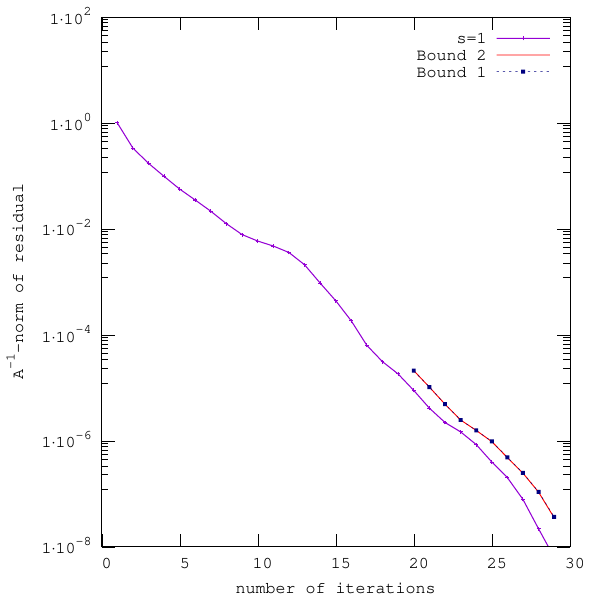}
  \caption{$k_1=1$, $k_2=0$ and $s=1$.} 
  \label{fig:10a}
  \end{subfigure}
   \begin{subfigure}{0.50\textwidth}
   \centering
    \includegraphics[width=0.9\textwidth]{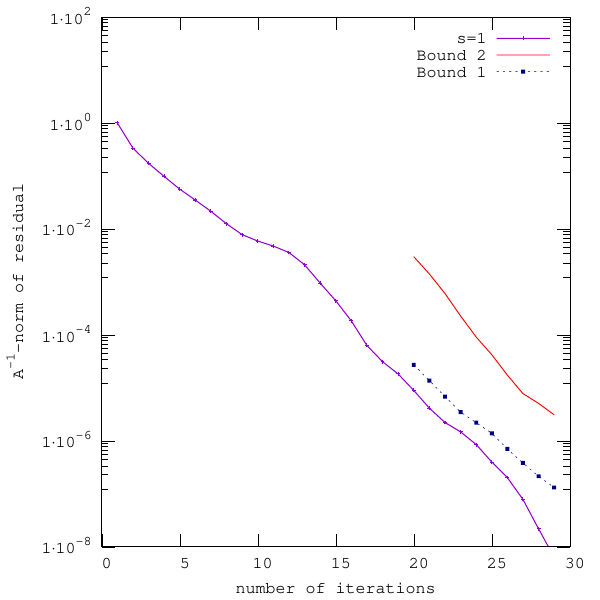}
   \caption{$k_1=2$, $k_2=0$ and $s=1$ .} 
   \label{fig:10b}
  \end{subfigure}
     \begin{subfigure}{0.48\textwidth}
   \centering
   \includegraphics[width=0.9\textwidth]{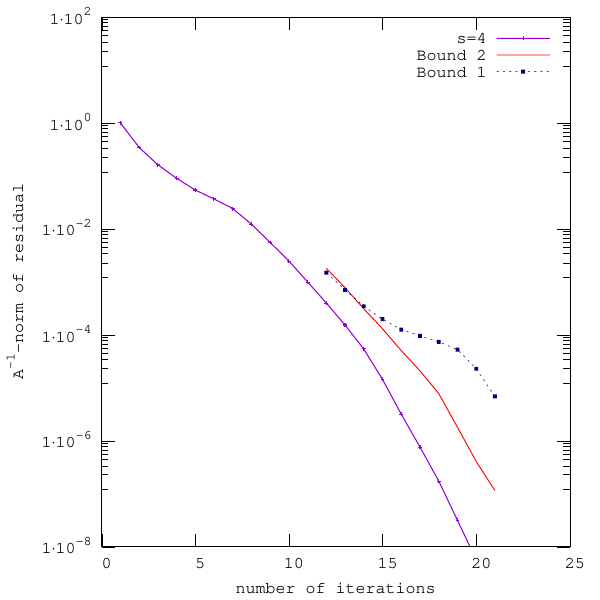}
   \caption{$k_1=4$, $k_2=0$ and $s=4$.} 
   \label{fig:10c}
   \end{subfigure}
     \begin{subfigure}{0.50\textwidth}
   \centering
    \includegraphics[width=0.9\textwidth]{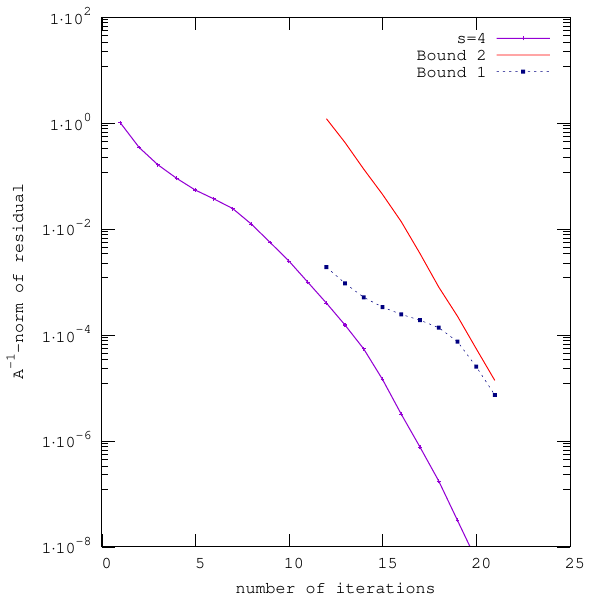}
   \caption{$k_1=5$, $k_2=0$ and $s=4$ .} 
   \label{fig:10d}
   \end{subfigure}
   \caption{Example  \ref{ex:ex_real_matrix}. 
   Poisson matrix preconditioned using an incomplete Cholesky factorization. 
   Residual convergence behavior for the block CG, 
   the bound $b_{2,j}$ corresponding to expression  (\ref{cota:experimentalJ}) 
     and {$ \left\lVert \bar{\bf R}_j \right\rVert_{A^{-1}}$}.  
   (a) Parameters $k_1=1$  and $\alpha_{20,1,0}=1.01$. 
   (b) Parameters $k_1=2$ and $\alpha_{20,2,0}=89.4478$. (c) Parameters $k_1=4$ and $\alpha_{10,4,0}=1.448558$. (d) Parameters $k_1=7$ and $\alpha_{10,7,0}=29.224$.}
     \label{fig:ex_realmatrix1}
 \end{figure}

\begin{example}\label{ex:ex_real_matrix}
{\rm
For our last example we consider a
$400\times 400$ matrix obtained with a standard discretization of the
2D Poisson equation, presonditioned with  incomplete Cholesky factorization with no fill.
Thus, the coefficient matrix is $\tilde{A}=L^{-1}AL^{-T}$. 
The maximum eigenvalue is 1.2015 and the minimum is 0.0724. The ten smallest eigenvalues are $0.0724$, $ 0.1652$, $0.1699$, $0.2483$, $0.2971$, $0.2994$, $0.3486$, $0.3742$, $0.4362$, $0.4367$, $0.4396$, $0.4802$. The rest of the eigenvalues are distributed between $0.5014$ and $1.2015$.

Figure \ref{fig:ex_realmatrix1} shows the behavior of the residual, together with  
the subspace and spectral  bounds $b_{1,j}$ and $b_{2,j}$ as in (\ref{cota:experimentalJ}).
In one set of experiments, we consider
 the block size $s=1$, and parameters $k_1=1$ and $k_2=0$, and $k_1=2$ and $k_2=0$. 
Figure~\ref{fig:ex_realmatrix1}(a) shows a good approximation of the bounds to the block CG residual.  
This is in accordance with the observation made that the convergence is mainly due to the convergence of 
the first eigenvalue. It can also be observed (similar to Example \ref{ex:ex_eigen_rep}) that when $k_1$ is 
larger than $s$, the bounds capture the slope of the residual but the spectral bound $b_{2,j}$ is 
far from  sharp (see Figure \ref{fig:ex_realmatrix1}(b)). 
Similar observation can be done for a larger block sizes, for instance, 
Figures \ref{fig:ex_realmatrix1}(c) and \ref{fig:ex_realmatrix1}(d) for values 
of $k_1=4$, $k_2=0$ and $s=4$, and  $k_1=5$, $k_2=0$ and $s=5$, respectively. 

Finally, comparing Figures \ref{fig:ex_realmatrix1}(c) and \ref{fig:ex_realmatrix1}(d), since more eigenvalues are captured when the block size is increased, the onset  of the superlinear convergence occurs earlier. 
Again, note the different horizontal scale of these two figures for $s=4$ as compared to those with $s=1$.
}
\end{example}

\section{Conclusions}
\label{sec:conclusions}

We extended the {\it a posteriori} spectral bound introduced by van der Sluis and 
van der Vorst \cite{vanderSluis1986} to the block CG case. 
We have also implemented the subspace bound introduced by Simoncini and Szyld \cite{Simoncini2005}. 
Numerical experiments show that both  bounds capture the slope of the residual 
after the onset of superlinearity.
In addition, when there is a cluster of eigenvalues in the lower part of the matrix spectrum, 
the spectral bound captures better the superlinearity even at the near the onset of superlinearity,
while the subspace bound is sharper in presence of repeated eigenvalues.  

Analyzing the bounds and the residual behavior of the block CG method, it can be  observed that 
the block method accelerates the convergence because it captures more eigenvalues. 
Therefore, in the presence of a cluster or repeated eigenvalues in the lower part of the spectrum, 
the larger the block size $ s $, the earlier is the onset of the superlinear convergence. Hence 
in these cases we suggest the use of block CG with the block size of the order of to the size of the cluster.

\bibliographystyle{siam}
\bibliography{references}

\end{document}